\documentclass[a4paper,10pt]{amsart}
\usepackage{amsmath,amsthm,amssymb,amsfonts,enumerate,color}
\usepackage[pdftex]{graphicx}

\newcommand{\R}{\mathbb{R}}

\newcommand{\N}{\mathbb{N}}

\newcommand{\s}{\sigma}
\renewcommand{\a}{\alpha}

\newcommand{\C}{\mathbb{C}}
\newcommand{\sgn}{\operatorname{sgn}}

\newtheorem{thm}{Theorem}[section]
\newtheorem{prop}[thm]{Proposition}
\newtheorem{cor}[thm]{Corollary}
\newtheorem{lem}[thm]{Lemma}

\theoremstyle{definition}
\newtheorem{defn}[thm]{Definition}
\newtheorem{rem}[thm]{Remark}

\numberwithin{equation}{section}

\allowdisplaybreaks

\author[A. Bernardis]{Ana Bernardis}

\author[F. J. Mart\'in-Reyes]{Francisco J. Mart\'in-Reyes}

\address{Departamento de An\'alisis Matem\'atico\\
         Universidad de M\'alaga\\
         29071 M\'alaga, Spain}
\email{ana.bernardis@gmail.com, martin\_reyes@uma.es}

\author[P. R. Stinga]{Pablo Ra\'ul Stinga}

\address{Department of Mathematics\\
Iowa State University\\
396 Carver Hall, Ames\\
IA 50011, United States of America}
\email{stinga@iastate.edu}

\author[J. L. Torrea]{Jos\'e L. Torrea}
\address{Departamento de Matem\'aticas\\
         Universidad Aut\'onoma de Madrid\\
         and\\
         ICMAT-CSIC-UAM-UCM-UC3M\\
         28049 Madrid, Spain}
\email{joseluis.torrea@uam.es}

\thanks{The first author was partially supported by Consejo Nacional de Investigaciones Cient\'{\i}ficas y T\'ecnicas and Universidad Nacional del Litoral (Argentina). The first and second authors were supported by grant MTM2011-28149-C02-02 from Spanish Government (Ministerio de Econom\'{\i}a y Competitividad) and grant FQM-354 from Junta de Andaluc\'{\i}a.
The third and fourth authors were supported by grant MTM2011-28149-C02-01 from Spanish Government (Ministerio de Econom\'{\i}a y Competitividad).}

\keywords{Nonlocal equations, maximum principle, extension problem, inversion, one-sided spaces}

\subjclass[2010]{Primary: 35R11, 34A08. Secondary: 26A33, 35A08, 35B50}

\begin{document}

\title[Nonlocal one-sided equations]{Maximum principles, extension problem and inversion \\
for nonlocal one-sided equations}

\begin{abstract}
We study one-sided nonlocal equations of the form
$$\int_{x_0}^\infty\frac{u(x)-u(x_0)}{(x-x_0)^{1+\alpha}}\,dx=f(x_0),$$
on the real line. Notice that to compute this nonlocal operator of order $0<\alpha<1$ at a point $x_0$
we need to know the values of $u(x)$ to the right of $x_0$, that is, for $x\geq x_0$.
We show that the operator above corresponds to a
fractional power of a one-sided first order derivative.
Maximum principles and a characterization with an extension problem in
the spirit of Caffarelli--Silvestre and Stinga--Torrea are proved. It is also shown that these fractional
equations can be solved in the general setting of weighted one-sided spaces.
In this regard we present suitable inversion results.
Along the way we are able to unify and clarify
several notions of fractional derivatives found in the literature.
\end{abstract}

\maketitle

\section{Introduction}

We analyze equations of the form
\begin{equation}\label{u}
\int_{x_0}^\infty\frac{u(x)-u(x_0)}{(x-x_0)^{1+\alpha}}\,dx=f(x_0),
\end{equation}
on $\R$. Expressions like the nonlocal operator above are in general connected with different notions of \textit{fractional} derivatives.  If the name ``derivative'' is reasonable, the object defined in \eqref{u} should satisfy,
in our opinion,  some 
of the fundamental properties of the true derivative. Even more, it would be desirable  to see the equation in (\ref{u}) as a certain limit of a classical local differential equation. If that is possible, then the  theory of partial differential equations could be applied to the classical equation and then obtain as  a consequence some properties for the fractional derivative in \eqref{u}. 
Finally, one of the important tasks would be to find spaces in which we can solve the equation \eqref{u}. In other words,  from the point of view of operator theory, something should be said 
about the inverse operator $f\rightarrow u.$ Along this paper all these questions are treated. In this flow of ideas, we  establish some maximum principles, see Theorem \ref{thm:A} and Corollary \ref{cor:1}, we show that the fractional derivative defined above is a Dirichlet-to-Neumann operator of a classical local PDE equation, see Theorem \ref{thm:B}, and  finally we solve the equation in some Lebesgue spaces related with the one-sided nature of the expression (\ref{u}), see Theorems \ref{thm:C} and \ref{thm:D}. 

Obviously one of our primary duties is to locate the operator in a framework for which the name ``fractional  derivative'' has sense. In order to do that in a reasonable way let us make some discussions about expressions like \eqref{u}. 

The expression $d^n y/dx^n$ was introduced  by G. W. Leibniz to denote derivatives of higher
integer order.  A natural thought  has been to extend the definition to
non integers values of $n$.  In September 1695, G. F. Antoine, Marquis de 
 L'H\^opital,  wrote a letter to  Leibniz asking ``What if $n$ be $1/2$?''.
This letter and  Leibniz's answer   are considered  the starting point of \textit{fractional calculus}, see \cite{Ros}.
Since then a lot of effort has been devoted in order to define and apply fractional derivatives
and fractional integrals. It is interesting to notice that different notions
of fractional derivatives and integrals have been used in Physics. For example in 1823,
N. H. Abel  used  fractional operations in the  formulation of the tautochrone problem, see \cite{Ros}.

The 19th century witnessed a lot of activity in the area.
The important contribution of Liouville, together with the names of Riemann and Weyl,
are constantly present in the theory of
fractional calculus.
Along this paper we shall consider the following fractional integral operators
\begin{equation}\label{W}
W_\alpha f(x) = \frac1{\Gamma(\alpha)} \int_x^\infty \frac{f(t)}{(t-x)^{1-\alpha}}\,dt
\end{equation}
and 
\begin{equation}\label{R}
R_\alpha f(x)= \frac1{\Gamma(\alpha)} \int_{-\infty}^x  \frac{f(t)}{(x-t)^{1-\alpha}}\,dt.
\footnote{These operators are usually called 
the Weyl and Riemann--Liouville fractional integrals, respectively.}
\end{equation}

There is much less consensus with respect to the definition of fractional derivatives.
From our point of view, it seems natural to think that after accepting the notions of $W_\alpha$ 
 and $R_\alpha$ as good definitions for fractional integrals, the corresponding fractional derivatives
 should satisfy a sort of ``Fundamental  Theorem of Calculus''.
In other words, the  composition of a fractional integral and a fractional derivative should give the identity,
or, which is the same, one should be able to find $u$ in \eqref{u} by taking the inverse operator.

On the other hand, it is well known that for good functions $f$ we have the Fourier transform relation
$$\widehat{\frac{d^{2}}{d x^2}f}(\xi)=-|\xi|^2\widehat{f}(\xi).$$
Therefore a possible definition of fractional operators could be (and, in fact, it is in several contexts):
\begin{itemize}
\item Fractional derivative of order $\alpha$, $0<\alpha <1$, is the operator whose
Fourier transform, when acting on good enough functions, is given by $-|\xi|^{\alpha/2}\widehat{f}(\xi)$.
\item Fractional integral of order $\alpha$, $0<\alpha <1$, is the operator whose Fourier transform,
when acting on good enough functions, is given by $-|\xi|^{-\alpha/2}\widehat{f}(\xi)$.
\end{itemize}
In the 1950's and 1960's these Fourier transform considerations appeared to be rather important,
specially for the community working in the intertwining area of PDEs,
Harmonic Analysis and function (Sobolev) spaces. In fact some papers
related with these ideas can be considered today as part of the history of the subject.
We mention here the works by A. P. Calder\'on \cite{calderon},
E. M. Stein \cite{St1,St2} and E. M. Stein and A. Zygmund \cite{SZ}.

In the late 1960's a series of papers, see \cite{Heywood,Heywood2,kober,Okikiolu},
dealt with operators of fractional integral type and fractional derivative type.
The authors proved that for a certain range of $p$ the composition of the
fractional integral and a certain fractional derivative gives the identity in $L^p(\R)$. 

More recently, K. Andersen considered in \cite{A} the same kind of problem
but for functions $f$ belonging to some weighted Lebesgue space  $L^p(\R, w)$,
where $w$ is a weight in the $A_p$ Muckenhoupt class. 
He also studied the range of the fractional integral operators type 
when acting on functions in the weighted Lebesgue space.
Finally S. G. Samko, A. Kilbas and O. Marichev
have some very interesting discussions about these topics in their nice book \cite{Sa}.

When analyzing the papers cited above, it is not obvious for the reader why
the authors choose a particular definition of fractional derivative. 
Moreover, some non trivial constants
(related with the Gamma function and not always the same one) appear frequently.
Indeed, G. O. Okikiolu in \cite{Okikiolu}, H. Kober in \cite{kober}
and P. Heywood in \cite{Heywood} considered the following fractional integral
operator
\begin{equation}\label{kalpha}
k_\alpha f(x) =
\frac1{2 \Gamma(\alpha) \sin \frac{\pi \alpha}{2} } \int_{\mathbb{R}} \frac1{|x-y|^{1-\alpha}} f(y)\,dy.
\end{equation}
In a parallel order of ideas P. Heywood in \cite{Heywood, Heywood2} considered
\begin{equation}\label{kalphaH}
k_\alpha^H f
=\tfrac{1}{\pi}\Gamma(1-\alpha) \sin\tfrac{\pi\alpha}{2}\int_{\mathbb{R}} \frac1{|x-y|^{1-\alpha}} f(y)\,dy,
\end{equation}
and  
\begin{equation}\label{Halpha}
H_\alpha f(x) = \frac1{2\Gamma(\alpha) \sin \frac{\pi \alpha}{2} } \int_{\mathbb{R}}
\frac{f(t){\rm sign}(t-x)}{|t-x|^{1-\alpha}}\,dt
= \tfrac{1}{\pi}\Gamma(1-\alpha)\cos \tfrac{\pi \alpha}{2}\int_{\mathbb{R}}
\frac{f(t){\rm sign}(t-x)}{|t-x|^{1-\alpha}}\,dt.
\end{equation}
P. Heywood also defined in \cite{Heywood, Heywood2} the fractional derivative operators 
\begin{equation}\label{kmenosalpha}
k_{-\alpha}u(x)=\tfrac{1}{\pi}\Gamma(1+\alpha) \sin\tfrac{\pi \alpha}{2} 
\operatorname{P.V.}\int_{\mathbb{R}} \frac{u(t)-u(x) }{|t-x|^{1+\alpha}}\,dt,
\end{equation}
and
\begin{equation}\label{Hmenosalpha}
\begin{aligned}
H_{-\alpha} u(x) &= \tfrac{1}{\pi}\Gamma(1+\alpha) \cos \tfrac{\pi\alpha}{2}
\int_{0}^{\infty} \frac{u(x+t)-u(x-t)}{t^{1+\alpha}}\,dt \\
&= \tfrac{1}{\pi}\Gamma(1+\alpha) \cos \tfrac{\pi\alpha}{2} \int_{\mathbb{R}}
\frac{{\rm sgn}(t-x) u(t)}{|t-x|^{1+\alpha}}\,dt,
\end{aligned}
\end{equation}
where the integrals are understood as a principal value integral.
Obviously when dealing with boundedness of operators the constants are irrelevant,
but as far as we want to get solvability in the sense of
inversion results or a ``Fractional Fundamental Theorem of Calculus''
the constants play a fundamental role.

One of the aims of this note is to give a general and consistent approach
to the definition of fractional integrals and derivatives.
This will unify formulas \eqref{W} and \eqref{R}, together with those in \eqref{kalpha}--\eqref{Hmenosalpha}
and also with the definitions involving Fourier transform.
In our opinion the best machinery to clarify all these concepts is the semigroup language.
The key idea will be the application of the classical formulas
\begin{equation}\label{negacla}
\lambda^{-\alpha} = \frac1{\Gamma(\alpha)} \int_0^\infty e^{-t\lambda}\, \frac{dt}{t^{1-\alpha}}, \qquad  
\lambda^{\alpha} = \frac1{\Gamma(-\alpha)} \int_0^\infty( e^{-t\lambda}-1)\, \frac{dt}{t^{1+\alpha}},
\end{equation}
valid for $\lambda >0$, and their appropriate extensions
to complex parameters $i\lambda$, see Section \ref{powers}.  These formulas will 
allow us to define positive and negative powers of an operator $L$ by the expressions 
\begin{equation*}\label{negacla2}
L^{-\alpha} = \frac1{\Gamma(\alpha)} \int_0^\infty e^{-tL}\, \frac{dt}{t^{1-\alpha}}, \qquad  
L^{\alpha} = \frac1{\Gamma(-\alpha)} \int_0^\infty( e^{-tL}-1)\, \frac{dt}{t^{1+\alpha}},
\end{equation*}
where $e^{-tL}$ is the associated heat semigroup.  We will use these ideas to
define positive  and negative powers of the classical derivatives
and the Laplace operator  on the real line,  see \eqref{derivada2},  \eqref{derivada1}
and Remark  \ref{derivada7}.

Having enclosed the fractional derivatives and integrals into the frame of the semigroup language, 
we can take advantage of the method to highlight some properties of the fractional derivatives.

Next we present our main results.

The first two main statements are the maximum and comparison principles for fractional derivatives,
as well as uniqueness for the corresponding Dirichlet problem.
By $D_{\rm right}$ we denote the derivative \textit{from the right} at the point $x\in\R$, that is,
$$D_{\rm right}\varphi(x) = \lim_{t\rightarrow 0^+} \frac{\varphi(x)-\varphi(x+t)}{t},$$
for good enough functions $\varphi$. Observe that $D_{{\rm right}}$ equals the negative of the
lateral derivative $\frac{d}{dx^{+}}$ as usually defined in Calculus.
Our definition turns out to be the suited one when inverting the classical fractional integrals as
we will see. Then (see Subsection \ref{derivada1}),
$$(D_{\rm right})^\alpha \varphi(x) =  \frac1{\Gamma(-\alpha)}  \int_{x}^\infty
\frac{\varphi(t)-\varphi(x)}{(t-x)^{1+\alpha}}\,dt.$$

\begin{thm}[Maximum and comparison principles]\label{thm:A}
Let $\varphi $ be a function in the Schwartz class $\mathcal{S}$ such that
$\varphi(x_0) = 0$ for some $x_0$, and $\varphi(x)  \ge 0$, for $x\ge x_0$.
Then $(D_{\rm right})^\alpha\varphi(x_0) \le 0$. Moreover,
$(D_{\rm right})^\alpha\varphi(x_0) = 0$ if and only if $\varphi (x) = 0$ for all $x\ge x_0$.

Let $\varphi,\psi\in\mathcal{S}$ such that $\varphi(x_0) =  \psi(x_0) $ for some $x_0$,
and $\varphi(x) \ge \psi (x) $ for  $x \ge x_0$. Then 
$(D_{\rm right})^\alpha\varphi(x_0) \le (D_{\rm right})^\alpha\psi(x_0)$.
Moreover, $(D_{\rm right})^\alpha\varphi(x_0) = (D_{\rm right})^\alpha\psi(x_0)$
if and only if $\varphi(x) =  \psi(x) $ for all $x \ge x_0.$
\end{thm}

\begin{cor}[Dirichlet problem -- Uniqueness]\label{cor:1}
Let $a<b$ be two real numbers and $\varphi,\psi\in\mathcal{S}$.
\begin{enumerate}
\item Suppose that $\varphi$ solves
$$\begin{cases}
(D_{\rm right})^\alpha\varphi=f,&\hbox{in}~[a,b),\\
\varphi=0,&\hbox{in}~[b,\infty).
\end{cases}$$
If $f\geq0$ in $[a,b)$ then $\varphi\geq0$ in $[a,\infty)$.
\item If $(D_{\rm right})^\alpha\varphi\leq0$ in $[a,b)$ and $\varphi\leq0$ in $[b,\infty)$, then
$$\sup_{x>a}\varphi(x)=\sup_{x>b}\varphi(x).$$
\item If $(D_{\rm right})^\alpha\varphi\geq0$ in $[a,b)$ and $\varphi\geq0$ in $[b,\infty)$, then
$$\inf_{x>a}\varphi(x)=\inf_{x>b}\varphi(x).$$
\item If
$$\begin{cases}
(D_{\rm right})^\alpha\varphi\geq(D_{\rm right})^\alpha\psi,&\hbox{in}~[a,b),\\
\varphi\geq\psi,&\hbox{in}~[b,\infty),
\end{cases}$$
then $\varphi\geq\psi$ in $[a,\infty)$. In particular, we have uniqueness of the Dirichlet problem
\begin{equation}\label{Dirichlet problem}\begin{cases}
(D_{\rm right})^\alpha\varphi=f,&\hbox{in}~[a,b),\\
\varphi=g,&\hbox{in}~[b,\infty).
\end{cases}\end{equation}
\end{enumerate}
\end{cor}

Notice that, because of the nonlocal one-sided nature of the operator $(D_{\rm right})^\alpha$,
\eqref{Dirichlet problem} is the correct formulation of the Dirichlet problem, that is,
the \textit{boundary} condition must be in fact a global condition \textit{to the right} of $b$.

The next main statement shows that the fractional derivatives on the line 
are Dirichlet-to-Neumann operators for an extension degenerate PDE problem
in $\mathbb{R}\times(0,\infty)$. We reach the most general result by taking
data $f$  in a weighted $L^p(w)$ space, where $w$ satisfies the one-sided version $A_p^+$ (see \eqref{A_p} and \cite{S}) of the familiar  $A_p$ condition of Muckenhoupt. 
As in the previous paragraph, the appearance of a one-sided condition
on the weight is natural due to the one-sided nature of the operators $D_{\rm right}$
and $(D_{\rm right})^\alpha$.

\begin{thm}[Extension problem]\label{thm:B}
Let  $f\in L^p(w)$, $w \in A_p^+,\, 1<p<\infty$, see \eqref{A_p}. Then the function
$$U(x,t):=\frac{t^{2\a}}{4^\a\Gamma(\a)}
\int_0^\infty e^{-t^2/(4s)}f(x+s)\,\frac{ds}{s^{1+\a}},\quad x\in\R,~t>0,$$
is a classical solution to the extension problem
$$\begin{cases}
-D_{\rm right}U+\frac{1-2\a}{t}\,U_t+U_{tt}=0,& \hbox{in}~\R\times(0,\infty), \\
\lim_{t\to0^+}U(x,t)=f(x), & \hbox{a.e. and in}~L^p(w).
\end{cases}$$
Moreover, for $c_\a :=\frac{4^{\a-1/2}\Gamma(\a)}{\Gamma(1-\a)}>0$,
$$-c_\a\lim_{t\to0^+}t^{1-2\a}U_t(x,t)=(D_{\rm right})^\alpha f(x),\quad\hbox{in the distributional sense}.$$
\end{thm}      

This Theorem is new even for functions $f\in L^2(\mathbb{R})$. 
A parallel statement substituting $D_{\rm right}$ by $-\Delta$ and 
$w\in A_p^+$ by $w\in A_p$ can also be proved with the same kind 
of reasoning. This parallel result is well known for the case
$L^2(\mathbb{R})$ (see \cite{Caffarelli-Sil, Stinga-Torrea})
but is new in the weighted case of $L^p(\mathbb{R},w)$.

Next we turn to the ``Fractional Fundamental Theorem of Calculus", 
which can also be interpreted as an inversion result of the corresponding fractional integral.
In this order of ideas, 
let us describe here a small part of the paper by K. F. Andersen \cite{A}.
Consider the operator 
$$I_\alpha f(x) := \tfrac{1}{\pi}\Gamma(1-\alpha) \sin\tfrac{\pi \alpha}{2}
\int_{-\infty}^{\infty} \frac{f(t)}{|t-x|^{1-\alpha}}\,dt,\quad x \in \mathbb{R}.$$
Let $ w(x)$, $x \in \mathbb{R}$, be a nonnegative weight function in the class
$A_{p,q}$, $1<p< 1/\alpha$, $1/q = 1/p -\alpha$,  see \eqref{claseApq}. It is well known
that $I_\alpha$ applies $L^p(w^p)$ onto $L^q(w^q)$ if and only if $w\in A_{p,q}$, see
the paper by Muckenhoupt and Wheeden \cite{MW}.
The following statements are proved in \cite{A}.
\begin{itemize}
\item[(i)] If  $f\in L^p(w^p)$ then
\begin{equation*}\label{int1} f(x) = -\lim_{\delta\rightarrow 0^+}\tfrac{1}{\pi}\Gamma(1+\alpha)
\sin\tfrac{\pi \alpha}{2}\int_{|t-x|\ge \delta} \frac{I_\alpha f(t)-I_\alpha f(x)}{|t-x|^{1+\alpha}}\,dt,
\end{equation*}
pointwise almost everywhere and in the norm of $L^p(w^p)$.
\item[(ii)]Analogously, if $f\in L^p(w^p)$ then
\begin{equation}\label{int3}
f(x) = \lim_{\delta\rightarrow 0^+} \frac1{\Gamma(-\alpha)}
\int_{x+\delta}^\infty \frac{W_\alpha f(t)-W_\alpha f(x)}{(t-x)^{1+\alpha}}\,dt,
\end{equation}
pointwise almost everywhere and in the norm of $L^p(w^p).$
\end{itemize}
These conclusions suggest that the operator in \eqref{kmenosalpha} as well as 
the fractional operator $(D_{\mathrm{right}})^\alpha$, defined above as
\begin{equation}\label{int5}
(D_{\rm right})^\alpha u(x) =\lim_{\delta\rightarrow 0^+} \frac1{\Gamma(-\alpha)} 
\int_{x+\delta}^\infty \frac{u(t)-u(x)}{(t-x)^{1+\alpha}}\, dt,
\end{equation}
are appropriate definitions of fractional derivatives. 
The two results have an obvious parallel structure that should be clarified in a proper way.
Also the role of the different constants appearing in such a similar results should be understood.
Notice that a limit has to be taken in \eqref{int5} to account for functions $u$ that are not smooth.
As a consequence, the result of Andersen stated in \eqref{int3}
can be read as a kind of ``Fractional Fundamental Theorem of Calculus'':
$$ (D_{\rm right})^\alpha W_\alpha f = f,$$
valid almost everywhere and in the weighted
$L^p(w^p)$ norm for functions $f\in L^p(w^p)$.
This identity involves operators that have some  one-sided
behavior (one could say that the operators look only at the \textit{future} values 
of the function after the time point $x$).
However, Andersen's statement considers a class of functions which is ``blind''
for this lateral behavior. We improve Andersen's result by taking the more natural class of
nonnegative  lateral weights $A_{p,q}^+$, see \eqref{lateral} for the definition.
 We remind that $W_\alpha$ maps $L^p(w^p)$ onto $L^q(w^q)$ if and only if $w\in A_{p,q}^+$,
 see \cite{AS,MPT,MT}.
This class strictly contains the class considered by Andersen.
Moreover, these lateral weights will be sensible when considering
either $D_{\rm right } $  or   $D_{\rm left} $.  See Section \ref{Slateral} in which 
we shall prove the following two versions of the ``Fractional Fundamental Theorem of Calculus''.

\begin{thm}\label{thm:C}
Let $0<\alpha<1$,  $1<p< 1/\alpha$ and  $1/q=1/p-\alpha$. If $w\in A_{p,q}^+$ 
then for all $f\in L^p(w^p)$
$$f(x) = (D_{\rm right} )^\alpha (W_\alpha f)(x),$$
in the almost everywhere sense and in the $L^p(w^p)$-norm. 
\end{thm}

Furthermore, we prove a pointwise
inversion formula with a much weaker hypothesis.

\begin{thm}\label{thm:D}
Let $0<\alpha<1$,  $1<p<1/\a$ and  $1/q=1/p-\a$.
Let  $w$ be a nonnegative measurable function such that $w^p$ and $w^{-p^\prime}$
are locally integrable on $\R$. Assume that $W_\alpha f(x) <+\infty$ for almost every $x$,
for all $f\in L^p(w^p)$. Then
$$f(x)=(D_{\rm right} )^\alpha (W_\alpha f)(x),\quad\hbox{for almost every}~x\in\R.$$
\end{thm}

Finally, we also present a characterization of the range in Theorem \ref{Inversionlateral3}.

 As we said in some lines above, we also pursue to unify the (apparently) different operators
found in the literature. In Section \ref{unified}  
we shall deal with the operators considered in the series of papers 
\cite{Heywood,Heywood2,kober,Okikiolu}.

After the account we just did on the historical aspects of fractional derivatives,
we need to mention here that these kinds of derivatives are being intensively studied due to their multiple applications
in real world phenomena. Indeed, models involving several types of fractional derivatives
arise in Physics, Biology, Financial Mathematics and Elasticity, among many other topics. 
A list of references would be long and it is out of the scope of this paper. We just mention
here the work \cite{Allen-Caffarelli-Vasseur} and the special issue \cite{JCP}, see also the references therein.

The reader will notice that we work mainly with $D_{\mathrm{right}}$, which gives a fractional
derivative that sees the function to the right of the point or \textit{into the future}. By considering $D_{\mathrm{left}}$ we obtain
a fractional derivative that considers the values of the function to the left of the point
or \textit{from the past}, see Section \ref{powers}. The latter is 
sometimes called the Caputo or Marchaud fractional derivative. All our results are equally valid for those operators.

The paper is organized as follows.
In Section \ref{powers} we establish the numerical formulas that allow to define powers of operators.
In the same section we use them to define fractional (one-sided and two-sided) derivatives and integrals.
The proofs of the maximum principles are contained in Section \ref{MAXIMUM}.
In Section \ref{EXTENSION}, by using some subordination formulas of Poisson type, 
we prove the extension result announced in Theorem
\ref{thm:B}. In Section \ref{unified}, we apply the ideas developed in Section \ref{powers}
to make a tour among the different kind of fractional operators that we found
in the literature and that we can reinterpret with our language.
The Fractional Fundamental Theorem of Calculus is considered
in Section \ref{Slateral}, in which we prove Theorems \ref{thm:C} and \ref{thm:D},
together with some inversion and range results for the fractional integral.

\section{Powers of operators}\label{powers}

We begin this section by recalling the following two formulas related with the Gamma function:
\begin{equation}\label{defi1}
\Gamma(\alpha) = \int_0^\infty e^{-t} t^\alpha \,\frac{dt}{t}, \quad  
\Gamma(-\alpha) = \int_0^\infty \big(e^{-t}-1\big)\,\frac{dt}{ t^{1+\alpha}}, \quad  0 < \alpha <1.
\end{equation}
These absolutely convergent integrals can be interpreted also as integrals
along the complex path $\{z=t:  0<t< \infty\}$.
By using the  Cauchy Integral Theorem we are able to prove the following.

\begin{lem}\label{lem:gamma function}
Let $0<\a<1$ and $-\pi/2\leq \varphi_0\leq\pi/2$. Consider the ray
in the complex plane $\mathrm{ray}_{\varphi_0}:=\{z=re^{i\varphi_0}:0<r<\infty\}$. Then
\begin{equation}\label{negatcomplex}
\Gamma (\alpha) = \int_{ {\rm ray}_{\varphi_0} } e^{-z} z^\alpha \,\frac{dz}{z},\quad  \hbox{and} \quad 
\Gamma(-\alpha) = \int_{{\rm ray}_{\varphi_0}} ( e^{-z}-1)\,\frac{dz}{z^{1+\alpha}}.
\end{equation}
\end{lem}

\begin{proof}
We do the computation for the case $0<\varphi_0\leq\pi/2$. The other case is completely analogous.
Notice that $F(z)=e^{-z}z^{\a-1}$ is holomorphic for $z\neq0$. Let $0<\varepsilon<R$.
Consider an angular sector $\mathcal{C}$ in the first quadrant of the complex plane $\Re(z),\Im(z)>0$ 
of aperture $\varphi_0$, but truncated at $|z|=\varepsilon$
and $|z|=R$. The boundary of $\mathcal{C}$ is oriented counterclockwise
and is given by the union of the following paths:
$\gamma_1=\{z=t:\varepsilon\leq t\leq R\}$,
$\gamma_2=\{z=Re^{i\varphi}:0\leq\varphi\leq\varphi_0\}$, $\gamma_3=\{z=te^{i\varphi_0}:\varepsilon\leq t\leq R\}$,
and $\gamma_4=\{z=\varepsilon e^{i\varphi}:0\leq\varphi\leq\varphi_0\}$. By the Cauchy Theorem,
$\int_{\mathcal{C}}F(z)\,dz=0$. We first notice that
\begin{align*}
	\bigg|\int_{\gamma_4}F(z)\,dz\bigg| &= \bigg|-\int_0^{\varphi_0}e^{-\varepsilon e^{i\theta}}\varepsilon^\alpha
	e^{i\a\theta}\,\frac{d(e^{i\theta})}{e^{i\theta}}\bigg| \\
	&\leq \varepsilon^\alpha\int_0^{\pi/2}e^{-\varepsilon\cos\theta}\,d\theta\leq \varepsilon^\a\pi/2\to0,
\end{align*}
as $\varepsilon\to0$. Similarly, but using now that $\cos\theta\geq1-(2\theta)/\pi$
whenever $0\leq\theta\leq\pi/2$,
\begin{align*}
	\bigg|\int_{\gamma_2}F(z)\,dz\bigg| 
	&\leq R^\alpha\int_0^{\pi/2}e^{-R\cos\theta}\,d\theta 
	\leq R^\alpha e^{-R}\int_0^{\pi/2}e^{(2R/\pi)\theta}\,d\theta\leq CR^{\alpha-1}\to0,
\end{align*}
as $R\to\infty$. Thus, after taking limits, we get that the first two integrals in \eqref{defi1}
and \eqref{negatcomplex} coincide. 

Let us consider the function $G(z)=(e^{-z}-1)/z^{1+\a}$, which is holomorphic in $\mathcal{C}$.
By the Mean Value Theorem,
\begin{align*}
	\bigg|\int_{\gamma_4}G(z)\,dz\bigg| &\leq \varepsilon^{-\a}\int_0^{\varphi_0}|e^{-\varepsilon e^{i\theta}}-1|
	\,d\theta\leq \varphi_0\varepsilon^{1-\a}\to0,
\end{align*}
as $\varepsilon\to0$. Also,
\begin{align*}
	\bigg|\int_{\gamma_2}G(z)\,dz\bigg| &\leq R^{-\a}\int_0^{\varphi_0}|e^{-\varepsilon e^{i\theta}}-1|
	\,d\theta\leq 2\varphi_0R^{-\a}\to0,
\end{align*}
as $R\to\infty$. By using the Cauchy Integral Theorem and taking the limits as
$\varepsilon\to0$ and $R\to\infty$ we get the equality
between the second integrals in \eqref{defi1} and \eqref{negatcomplex}.
\end{proof}

\begin{cor}\label{formulas} Let $0<\alpha <1$ and $\lambda \neq 0$.  Then
\begin{equation}\label{nega}
(i\lambda)^{-\alpha} = \frac1{\Gamma(\alpha)} \int_0^\infty e^{-i\lambda t }t^{\alpha}\,\frac{dt}{t},\quad
\hbox{and}\quad(i\lambda)^{\alpha} = \frac1{\Gamma(-\alpha)} \int_0^\infty
(e^{-i\lambda t}-1)\,\frac{dt}{t^{1+\alpha}}.\end{equation}
\end{cor}

Needless to say, the identities in \eqref{negacla} follow from \eqref{defi1}.
We will use \eqref{negacla} and \eqref{nega} to
define the negative and positive fractional powers of the operators we are going to work with.

For further reference we list here the following well known identities for the Gamma function:
\begin{equation} \label{gamma1}
\Gamma(z) = \frac{\Gamma(\frac{z}{2}) \Gamma( \frac{z}{2} +\frac1{2}) } {2^{1- z} \sqrt{\pi}}, \quad 
\Gamma(1-z) \Gamma(z) = \frac{\pi} {\sin(\pi z)}.
\end{equation}

In the rest of the section we shall present the definitions of positive and negative powers
of the lateral derivatives on the line.

\subsection{Negative powers of first order derivatives on $\R$}

Let us begin with the following

\begin{defn} 
We call  {\bf derivative from the right} of the function $u$ at the point $x$ to the limit:
\begin{equation*}
D_{{\rm right}}u(x) = \lim_{t\rightarrow 0^+} \frac{u(x)-u(x+t)}{t},  
\end{equation*}
whenever it exists.
Analogously, the {\bf derivative from the left} of the function $u$ in the point $x$ is
\begin{equation*} D_{{\rm left}}u(x) = \lim_{t\rightarrow 0^+}
 \frac{u(x)- u(x-t)}{t}\, .  
\end{equation*}
\end{defn}

\begin{rem} 
{\it  For good enough functions we have the expected identities
\begin{equation*}\label{Fderivada}
\widehat{D_{{\rm right}} u} (\xi) = - i \xi \widehat{u}(\xi) \quad
\hbox{and}\quad\widehat{D_{{\rm left}} u} (\xi) =  i \xi \widehat{u}(\xi),
\end{equation*}
where by $\widehat{u}(\xi)$ we denote the Fourier transform of the function $u$, that is,
$$ \widehat{u}(\xi) = \frac{1}{(2\pi)^{1/2}}\int_{\R}u(x)e^{-ix\xi} \,dx,\quad \xi\in\R.$$}
\end{rem}

Hence, from \eqref{Fderivada} and Corollary \ref{formulas}, for good enough functions $f$ we can write
\begin{equation*}\label{derivada2}
\begin{aligned}
(D_{{\rm right}})^{-\alpha}f(x)&= \frac{1}{\Gamma(\alpha)} \int_0^\infty   f(x+t)t^\alpha\, \frac{dt}{t}=W_\alpha f(x);\\
(D_{{\rm left}})^{-\alpha}f (x) &= \frac{1}{\Gamma(\alpha)} \int_0^\infty   f (x-t) t^\alpha \,\frac{dt}{t}= R_\alpha f(x) .
\end{aligned}
\end{equation*}
It is clear from the Fourier transform definition that these operators are not bounded in $L^2(\R)$.

\subsection{Positive powers of first order derivatives on $\R$}

Parallel to the case of negative powers, we can 
use the Fourier transform identities \eqref{Fderivada} together
with Corollary \ref{formulas} to get the identities
\begin{equation*}
\begin{aligned}\label{derivada1}
(D_{{\rm right}})^{\alpha}u(x)&=\frac{1}{\Gamma(-\alpha)} \int_0^\infty\frac{ u(x+t)-u(x)}{t^{\alpha+1}}\, dt = 
 \frac{1}{\Gamma(-\alpha)} \int_x^\infty  \frac{u(t)-u(x)}{(t-x)^{\alpha +1}} \,dt; \\
 (D_{{\rm left}})^{\alpha}u(x) &= \frac{1}{\Gamma(-\alpha)} \int_0^\infty  \frac{u(x-t)-u(x)}{t^{\alpha +1}}\,dt
=\frac{1}{\Gamma(-\alpha)} \int_{-\infty}^x\frac{ u(t)-u(x)}{(x-t)^{\alpha+1}}\,dt,
\end{aligned}
\end{equation*}
valid for sufficiently smooth functions. Notice from the Fourier transform identities
that these operators do not preserve the Schwartz class $\mathcal{S}$, see Remark \ref{rem:Silvestre}.

\begin{rem}\label{rem:lateral}
{\it Observe that the local operator $D_{{\rm right}}$, when acting on a function
$u$ at a point $x$, depends only on the values of $u$ in an
arbitrarily small neighborhood to the \textit{right} of the point $x$. However the 
nonlocal operators $(D_{{\rm right}})^{\alpha}$ and $(D_{{\rm right}})^{-\alpha}$
depend on the values of $u$ on the whole half line $(x,\infty)$.
As we mentioned in the Introduction,
a parallel comment can be made about the powers of the operator $D_{{\rm left}}.$}
\end{rem}

\begin{rem}\label{rem:Silvestre}
{\it Since $0<\alpha<1$, it is obvious that the integrals in 
\eqref{derivada1} are absolutely convergent for functions
in the Schwartz class $\mathcal{S}$. On the other hand, it is clear that for $\varphi,\psi \in \mathcal{S}$,
\begin{align*}
\Gamma(-\a) \int_{\mathbb{R}} (D_{{\rm right}})^{\alpha}\varphi(x) \psi(x)\,dx
&=\lim_{\varepsilon\to0^+}\int_\varepsilon^\infty\int_{\mathbb{R}}  
\frac{\varphi(x+t)-\varphi(x)}{t^{1+\alpha}}\,\psi(x)\,dx\,dt \\
&= \lim_{\varepsilon \rightarrow 0^+} \Bigg(\int_\varepsilon^\infty \int_{\mathbb{R}}
\frac{\varphi(x+t)}{t^{1+\alpha}}\,\psi(x) \,dx\, dt -
\int_\varepsilon^\infty \int_{\mathbb{R}}\frac{\varphi(x)}{t^{1+\alpha}}\, \psi(x)\,dx\,dt \Bigg) \\
&=\lim_{\varepsilon \rightarrow 0} \Bigg( \int_\varepsilon^\infty \int_{\mathbb{R}}
\frac{\varphi(y)}{t^{1+\alpha}}\,\psi(y-t)\, dy\, dt - \int_\varepsilon^\infty
\int_{\mathbb{R}}\frac{\varphi(x)}{t^{1+\alpha}}\, \psi(x)\,dx\,dt \Bigg) \\
&=\lim_{\varepsilon \rightarrow 0} \int_\varepsilon^\infty \int_{\mathbb{R}}
\frac{\varphi(y)}{t^{1+\alpha}}\,\big(\psi(y-t)-\psi(y)\big)\,dy\,dt \\
&={\Gamma(-\a)}\int_{\mathbb{R}}\varphi(x)  (D_{{\rm left}})^{\alpha}\psi(x)\,dx.
\end{align*}
It is easy to check that for $\varphi\in \mathcal{S} $, $(D_{{\rm right }})^{\alpha} \varphi \in\mathcal{S}_\alpha $, where
$$\mathcal{S}_\alpha:=\bigg\{f\in C^{\infty}(\R):(1+|x|^{1+\a})f^k(x)\in L^\infty(\R),~
\hbox{for each}~k\geq0\bigg\}.$$
The topology in $\mathcal{S}_\a$ is given by the family of seminorms
$[f]_k:=\sup_{x\in\R}\big|(1+|x|^{1+\alpha})f^{(k)}(x)\big|$, for $k\geq0$.
Let $\mathcal{S}_\alpha'$ be the dual space of $\mathcal{S}_\alpha$.
Then the symmetry showed above allows us to extend the definitions of $(D_{{\rm right }})^{\alpha}$
and $(D_{{\rm left }})^{\alpha}$ to the space $\mathcal{S}_\alpha'$ by duality. Namely,
if $h\in\mathcal{S}_\a'$ then $(D_{{\rm right}})^\alpha h$ is the tempered distribution given by
$$\langle (D_{{\rm right}})^\alpha h,\varphi\rangle=\langle
h,(D_{{\rm left}})^\alpha\varphi\rangle,\quad\hbox{for every}~\varphi\in\mathcal{S}.$$
Moreover, $(D_{{\rm right}})^\alpha$ is a continuous operator from $\mathcal{S}_\a'$
into $\mathcal{S}'$.
In particular, the operators will be defined in the space of functions
$$L_\alpha:=\bigg\{h:\R\to\R:\|h\|_{L_\alpha}:=\int_{\R}\frac{|h(x)|}{1+|x|^{1+\alpha}}\,dx< \infty\bigg\}.$$
See \cite{silvestre} for a parallel reasoning in the case of
the fractional Laplacian $(-\Delta)^\alpha$ on $\R^n$.}
\end{rem}

\section{Maximum principles for fractional one-sided derivatives}\label{MAXIMUM}

In this section we prove Theorem \ref{thm:A} and Corollary \ref{cor:1}.

\begin{proof}[Proof of Theorem \ref{thm:A}]
We observe that, under the hypotheses listed in the statement,
$$(D_{\rm right})^\alpha \varphi(x_0)=
\frac1{\Gamma(-\alpha)}\int_0^\infty \frac{\varphi(x_0+t)}{t^{1+\alpha}} \,dt,$$
which is obviously nonpositive.
Moreover, if $(D_{\rm right})^\alpha \varphi(x_0) =0$
then $\varphi(x_0+t)$ must be zero for every $t>0$. The comparison principle follows
by considering $\varphi-\psi$.
\end{proof}

\begin{proof}[Proof of Corollary \ref{cor:1}]
For (1), by contradiction, suppose that there
is a point $x_0\in[a,b)$ where $\varphi$ attaines a global negative minimum.
Then by Theorem \ref{thm:A}
we have $(D_{\rm right})^\alpha\varphi(x_0)\leq0$. Now we have two cases.
If $(D_{\rm right})^\alpha\varphi(x_0)=0$
then necessarily $\varphi(x)=\varphi(x_0)<0$ for all $x\geq x_0$, contradicting the hypothesis that
$\varphi(x)=0$ for $x\geq b$. The case $(D_{\rm right})^\alpha\varphi(x_0)<0$
also leads to a contradiction because, by assumption, $f(x_0)\geq0$.
Therefore $\varphi\geq0$ in $[a,\infty)$.

Notice that (3) follows from (2) by considering $-\varphi$.
For (2), again by contradiction, suppose that the $\sup_{x>a}\varphi(x)$
is not attained in $[b,\infty)$. Then there exists $x_0\in[a,b)$ such that $\varphi(x_0)$
is a global maximum of $\varphi$ on $[a,\infty)$. Hence, by
Theorem \ref{thm:A}, $(D_{\rm right})^\alpha\varphi(x_0)\geq0$.
The case $(D_{\rm right})^\alpha\varphi(x_0)>0$ contradicts the fact that $(D_{\rm right})^\alpha\varphi\leq0$
in $[a,b)$. Now if $(D_{\rm right})^\alpha\varphi(x_0)=0$ then,
by Theorem \ref{thm:A}, $\varphi(x)=\varphi(x_0)$ for all $x\geq x_0$,
so that the $\sup_{x>a}\varphi(x)$ is attained in $[b,\infty)$, again a contradiction with our
initial assumption. Hence the $\sup_{x>a}\varphi(x)$ is attained in $[b,\infty)$.

Finally (4) is a consequence of (3) applied to $\varphi-\psi$. The uniqueness is immediate.
\end{proof}

\section{Extension problem for fractional one-sided derivatives}\label{EXTENSION}

In order to prove Theorem \ref{thm:B}, we need some preparation and notation.
Fix $0<\a<1$.  Given a semigroup $\{T_t\}_{t\geq0}$ acting on real functions,
the \textit{generalized Poisson integral} of $f$ is given by
\begin{equation}\label{gen Poisson}
P_t^\a f(x)=\frac{t^{2\alpha}}{4^\alpha\Gamma(\alpha)}
\int_0^\infty e^{-t^2/(4s)}T_sf(x)\,\frac{ds}{s^{1+\alpha}},\quad x\in\R,
\end{equation}
see \cite[(1.9)]{Stinga-Torrea}. In the case $\a=1/2$, $P_t^{1/2}f$
is the Bochner subordinated Poisson semigroup of $T_t$, see \cite[Chapter~II,~Section~2]{Stein}.
For this special case we write $P_tf\equiv P_t^{1/2}f$.

If we consider the semigroup of translations $T_sf(x) = f(x+s)$, $s\geq0$, then
$P_t^\a f(x)=f\ast k_\varepsilon(x):=\int_\R
f(s)k_\varepsilon(x-s)\,ds$, with 
$$k(x):=\frac{e^{-1/(4(-x))}}{4^\alpha\Gamma(\alpha)(-x)^{1+\alpha}}
\chi_{(-\infty,0)}(x),$$
 $k_\varepsilon(x)=\frac1\varepsilon k(x/\varepsilon)$ and  $\varepsilon=t^2$.
 Since $k$ is increasing and integrable in $(-\infty, 0)$, it is well known that 
 $$k^*f(x)=\sup_{\varepsilon>0}|f|\ast k_\varepsilon(x)=\int_\R
|f(t)|k_\varepsilon(x-t)\,dt,$$
is pointwise controlled by the usual Hardy Littlewood maximal operator. However, since the support of $k$ is 
$(-\infty, 0)$, a sharper control can be obtained by using the  one-sided Hardy-Littlewood maximal operator. This control and the behavior of $k^*$ in weighted $L^p$-spaces will be used in the results of this paper. In the next subsection, we revise briefly some of the results we shall use.

\subsection{Approximations of the identity and lateral weights}\label{approx}
Let $k$ be any nonnegative integrable function with support in
$(-\infty, 0)$, increasing in $(-\infty, 0)$. Define
$k_\varepsilon$, $\varepsilon>0$, and 
$k^*f$ as before
when $f$ is locally integrable. Then (see \cite{L})
\begin{equation}\label{identity2}
k^*f(x)\leq\left( \int_\R k\right)M^+f(x),\quad\hbox{for a.e.}~x\in\R,
\end{equation}
where the one-sided Hardy-Littlewood maximal function $M^+$ is defined as
$$M^+f(x)=\sup_{h>0}\frac{1}{h}\int_x^{x+h}|f(t)|\,dt.$$
If $\omega$ is a nonnegative measurable function on $\R$ and $1<s<\infty$ then
$M^+:L^s(\omega)\to L^s(\omega)$ is bounded if and only if $\omega$ satisfies the one-sided
Muckenhoupt $A_{s}^+$ condition (see \cite{S}); that is, there exists  $C>0$ such that
\begin{equation}\label{A_p}
\left(\frac{1}{h} \int_{a-h}^a \omega  \right)^{1/s}\left(\frac{1}{h} \int^{a+h}_a \omega^{1-s'}  \right)^{1/s'} \leq C,
\end{equation}
for all real numbers $a$ and all $h>0$, where $s+s^\prime=ss^\prime$. It is clear that for every $a\in\R$,   $\omega(x)\in A_{s}^+\Leftrightarrow \omega(x+a)\in A_{s}^+$. Changing the orientation of the real line, we can define $M^-$ and the corresponding $A_{s}^-$ condition. Notice that $\omega(x)\in A_{s}^+\Leftrightarrow \omega(-x)\in A_{s}^-\Leftrightarrow \omega^{1-s'}\in A_{s^\prime}^- $.
In the limit case $p=1$, we have that $M^+$ is bounded from  $L^1(\R,w)$ into weak-$L^1(\R,w)$ if and only if  $w \in A_1^+$, that is, there exists $C>0$ such that $M^-w\leq Cw$ a.e.

It follows from \eqref{identity2} that if $\omega$ and $v$ are positive measurable  weights,
$v$ is locally integrable and $M^+:L^s(v)\to L^s(\omega)$ is bounded, for $1<s<+\infty$, and $k$ is as before,
then $k^*:L^s(v)\to L^s(\omega)$ is bounded and 
\begin{equation}\label{aprox}
\lim_{\varepsilon\to 0^+}f\ast k_\varepsilon(x)=\left( \int_\R k\right)f(x).
\end{equation}
almost everywhere and in the norm of $L^s(\omega)$ for all $f\in L^s(v)$. The same results hold if the   kernel is a real valued function such that
$|k|$ is dominated by an integrable function $\tilde k$ with support in
$(-\infty, 0)$ and increasing in $(-\infty, 0)$; the only difference is that we have 
$ \int_\R \tilde k$ instead of $ \int_\R k $.

\begin{rem}\label{positivity}
{\it From the definition we have the following:
if $\omega\in A_s^+$ then there exist $a$ and $b$, $-\infty\leq a\leq b\leq \infty$ such that
$\omega= 0$ in $(-\infty , a)$, $\omega=\infty$ in $(b,\infty )$, $0<\omega<\infty$ in $(a,b)$, $\omega\in L_{loc}^1(a,b)$ and,
if $1<s<\infty$,  $\omega^{1-s'}\in L_{loc}^1(a,b)$ (see \cite{Lai}).
Then, when working with one-sided weights, we can assume without loss of generality that $(a,b)=\R$. Throughout the paper, we assume that $0<\omega<\infty$ in $\R$.}
\end{rem}

\begin{rem}\label{property1}
{\it It follows from  \cite[Lemma 4,p. 540]{MT} 
that if $\omega\in A_s^+$ then for all $N>0$,
$$\int_{x>a+N} \frac{\omega^{1-s^\prime}(x)}{(x-a)^{p^\prime}}\,dx<\infty.$$}
\end{rem}

\subsection{The extension problem}
We are ready to state and prove a Theorem which contains
Theorem \ref{thm:A} of the Introduction.

\begin{thm}\label{STo}
Consider the semigroup of translations $T_tf(x) = f(x+t)$, $t\geq0$,
initially acting on functions $f\in\mathcal{S}$.
Let $P_t^\alpha f$, $0<\alpha<1$, be as in \eqref{gen Poisson}. Then:
 \begin{enumerate}[$(1)$]
 \item $P_t^\a$ is a bounded linear operator from $L^p(\R)$ into itself for $1\leq p\leq\infty$,
and $\|P_t^\a f\|_{L^p(\R)}\leq\|f\|_{L^p(\R)}$.
  \item When $f\in\mathcal{S}$, the Fourier transform of $P_t^\alpha f$ is given by 
  $$\widehat{P_t^\a f}(\xi)=\frac{2^{1-\a}}{\Gamma(\a)}\,(-it\xi^{1/2})^{\a}
\mathcal{k}_{\a}(-it\xi^{1/2})\widehat{f}(\xi),\quad\xi\in\R,$$
  where $\mathcal{k}_\nu(z)$ is the modified Bessel function of the third kind or Macdonald's function,
  which is defined for arbitrary $\nu$ and $z\in\C$, see \cite[Chapter~5]{Lebedev} and \eqref{identity}.
  In particular,
  $$\widehat{P_tf}(\xi)=e^{-t(-i\xi)^{1/2}}\widehat{f}(\xi).$$
  \item The maximal operator $P^{\alpha}_\ast f(x) = \sup_{t>0}|P_t^\alpha f(x)|$
  is bounded from $L^p(\R,w)$ into itself, for $w \in A_p^+, 1<p<\infty$,
  and from $L^1(\R,w)$ into weak-$L^1(\R,w)$, for $w \in A_1^+.$
  \item Let  $f \in L^p(w)$, for $w\in A_p^+$, $1\le p<\infty$.
  The function $U(x,t)\equiv P_t^\alpha f(x)$ is a classical solution to the extension problem
  $$\begin{cases}
     -D_{\rm right}U+\frac{1-2\a}{t}\,U_t+U_{tt}=0, & \hbox{in}~\R\times(0,\infty), \\
      \lim_{t\to0^+}U(x,t)=f(x), & \hbox{a.e and in}~L^p(w).
   \end{cases}$$
   Moreover, if $c_\a=\frac{4^{\a-1/2}\Gamma(\a)}{\Gamma(1-\a)}>0$ then
   \begin{equation}\label{limit}
   -c_\alpha\lim_{t\to0}\big(t^{1-2\alpha}U_t\big)=(D_{\rm right})^\alpha f,\quad\hbox{in the sense of distributions}.
   \end{equation}
 \end{enumerate}
\end{thm}

\begin{proof} 
To begin with,
observe that, by using the change of variables $t^2/(4s)=r$ in \eqref{gen Poisson},
\begin{equation}\label{Poisson cambio}
P_t^\a f(x)=\frac{1}{\Gamma(\a)}\int_0^\infty e^{-r}T_{t^2/(4r)}f(x)\,\frac{dr}{r^{1-\alpha}}.
\end{equation}
By applying Minkowski's integral inequality we have 
\begin{align*}
	\|P_t^\a f\|_{L^p(\R)} &\leq \frac{1}{\Gamma(\a)}\int_0^\infty e^{-r}
	\|T_{t^2/(4r)}f\|_{L^p(\R)} \,\frac{dr}{r^{1-\alpha}}\leq \|f\|_{L^p(\R)}.
\end{align*}

 Let us continue with (2). From \eqref{Poisson cambio} we can readily see that
\begin{equation}\label{Fourier Pt}
\widehat{P_t^\a f}(-\xi)= \Bigg(\frac{1}{\Gamma(\a)}\int_0^\infty e^{-r}
e^{-i\xi t^2/(4r)}\,\frac{dr}{r^{1-\a}}\Bigg)\widehat{f}(-\xi)=:H(t,\xi)\widehat{f}(-\xi).
\end{equation}
To relate the Fourier multiplier $H(t,\xi)$ with the Bessel function
$k_\a$ let us define, for $\varepsilon>0$,
$$H_\varepsilon(t,\xi)=\frac{1}{\Gamma(\a)}
\int_0^\infty e^{-r}e^{-(i\xi+\varepsilon)t^2/(4r)}\,\frac{dr}{r^{1-\a}}\quad t>0,~\xi\in\R.$$
It is clear that, by dominated convergence, $H_\varepsilon(t,\xi)\to H(t,\xi)$ as $\varepsilon\to0$,
for each $t,\xi$. Recall the following identity, see \cite[p.~119]{Lebedev},
\begin{equation}\label{identity}
\mathcal{K}_\nu(z)=\frac{1}{2}\Big(\frac{z}{2}\Big)^\nu\int_0^\infty e^{-r-z^2/(4r)}\,\frac{dr}{r^{1+\nu}},\quad|\arg z|<\pi/4,
\end{equation}
valid for arbitrary $\nu$. In \eqref{identity} we choose $\nu=-\a$ and
$z^2=z_\varepsilon^2=(i\xi+\varepsilon)t^2$. We have
$$z_\varepsilon=|(\xi^2+\varepsilon^2)t^4|^{1/4}e^{i\frac{\arg(i\xi+\varepsilon)}{2}}.$$
When $\xi>0$ we choose the argument above to be $0<\arg(i\xi+\varepsilon)<\pi/2$
and when $\xi<0$ we take $-\pi/2<\arg(i\xi+\varepsilon)<0$; thus $|\arg z_\varepsilon|<\pi/4$.
Therefore identity \eqref{identity} can be applied to this choice of $z=z_\varepsilon$.
The definition of $H_\varepsilon$ and the fact that $\mathcal{K}_{-\a}=\mathcal{K}_{\a}$ 
(see \cite[p.~110]{Lebedev}) then give
\begin{align*}
H_\varepsilon(t,\xi) &= \frac{2^{1-\a}}{\Gamma(\a)}\,z_\varepsilon^\a \mathcal{K}_{\a}(z_\varepsilon)
= \frac{2^{1-\a}}{\Gamma(\a)}\,|(\xi^2+\varepsilon^2)t^4|^{\a/4}e^{i\frac{\a\arg(i\xi+\varepsilon)}{2}}
\mathcal{K}_{\a}(z_\varepsilon) \\
&= \frac{2^{1-\a}}{\Gamma(\a)}\,|(\xi^2+\varepsilon^2)t^4|^{\a/4}
\big(\cos\tfrac{\a\arg(i\xi+\varepsilon)}{2}+i\sin\tfrac{\a\arg(i\xi+\varepsilon)}{2}\big)
\mathcal{K}_{\a}(z_\varepsilon).
\end{align*}
By taking the limit as $\varepsilon\to0$ in the last identity we get
\begin{align*}
H(t,\xi) &= \frac{2^{1-\a}}{\Gamma(\a)}\,z_0^\a \mathcal{K}_{\a}(z_0)
= \frac{2^{1-\a}}{\Gamma(\a)}\,|\xi|^{\a/2}t^\a 
e^{i\frac{\a\arg(i\xi)}{2}}\mathcal{K}_{\a}\Big(|\xi|^{1/2}te^{i\frac{\arg(i\xi)}{2}}\Big) \\
&= \frac{2^{1-\a}}{\Gamma(\a)}\,(t|\xi|^{1/2}e^{i\sgn\xi\frac{\pi}{4}})^\a
\mathcal{K}_{\a}\big(t|\xi|^{1/2}e^{i\sgn\xi\frac{\pi}{4}}\big) \\
&= \frac{2^{1-\a}}{\Gamma(\a)}\,(it|\xi|^{1/2}\sgn\xi)^{\a}\mathcal{K}_\a\big(it|\xi|^{1/2}\sgn\xi\big).
\end{align*}
The conclusion follows by replacing $\xi$ by $-\xi$ above and using \eqref{Fourier Pt}.
In particular, when $\a=1/2$, $\mathcal{K}_{1/2}(z_0)=(\frac{\pi}{2z_0})^{1/2}e^{-z_0}$,
see \cite[p.~112]{Lebedev}, so that
$H_{1/2}(t,\xi)=e^{-z_0}=e^{-t(i|\xi|\sgn\xi)^{1/2}}.$

For (3), by using (\ref{Poisson cambio}) we have
$$P^{\a}_\ast f(x)\le \frac{1}{\Gamma(\a)}\int_0^\infty e^{-r}\sup_{t>0}|T_{t^2/(4r)}f(x)|\,\frac{dr}{r^{1-\alpha}}
\le \|f\|_{L^\infty(\mathbb{R})}.$$
That is  $P^{\a}_\ast$ maps $L^\infty(\mathbb{R})$ into itself.  Note that, by calling $r=x+s$ in \eqref{gen Poisson},
$$P_t^\a f(x)=\frac{t^{2\alpha}}{4^\alpha\Gamma(\alpha)}
\int_x^\infty\frac{e^{-t^2/(4(r-x))}}{(r-x)^{1+\alpha}}f(r)\,dr=P_t^\a\ast f(x),$$
where the kernel is given by
$$P_t^\alpha(x):=\frac{t^{2\alpha}e^{-t^2/(4(-x))}}{4^\alpha\Gamma(\alpha)(-x)^{1+\alpha}}
\chi_{(-\infty,0)}(x).$$
A direct application of the results about approximations of the identity
and the characterization of $A_p^+$ presented in Subsection \ref{approx}
to the kernel $k_\alpha= P_1^\alpha(x)$ gives the complete statement (3).

Now we deal with (4). We compute:
\begin{align*}
	-D_{\rm right}ÊU(x,t) &= \lim_{h\to 0^+}\frac{U(x+h,t)-U(x,t)}{h} \\
	&= \frac{t^{2\alpha}}{4^\alpha\Gamma(\alpha)}\lim_{h\to 0^+}\frac{1}{h}
	\Bigg[\int_{x+h}^\infty\frac{e^{-t^2/(4(r-x-h))}}{(r-x-h)^{1+\alpha}}f(r)\,dr
	-\int_x^\infty\frac{e^{-t^2/(4(r-x))}}{(r-x)^{1+\alpha}}f(r)\,dr\Bigg] \\
	&= \frac{t^{2\alpha}}{4^\alpha\Gamma(\alpha)}\lim_{h\to 0^+}\frac{1}{h}
	\Bigg[\int_{x+h}^\infty\Bigg(\frac{e^{-t^2/(4(r-x-h))}}{(r-x-h)^{1+\alpha}}-
	\frac{e^{-t^2/(4(r-x))}}{(r-x)^{1+\alpha}}\Bigg)f(r)\,dr \\
	&\qquad-\int_x^{x+h}\frac{e^{-t^2/(4(r-x))}}{(r-x)^{1+\alpha}}f(r)\,dr\Bigg] .
\end{align*}
By the Mean Value Theorem,
$$\Bigg|\frac{e^{-t^2/(4(r-x-h))}}{(r-x-h)^{1+\alpha}}-
\frac{e^{-t^2/(4(r-x))}}{(r-x)^{1+\alpha}}\Bigg| \le 
Ch \sup_{0<\theta <1} \frac{e^{-t^2/(4(r-x-\theta h))}}{(r-x-\theta h)^{2+\alpha}}.$$
We can assume that $h<1/2$. Then,
\begin{align*}
\int_{x+h}^\infty  \Bigg| \frac1{h} &\Bigg(\frac{e^{-t^2/(4(r-x-h))}}{(r-x-h)^{1+\alpha}}-
	\frac{e^{-t^2/(4(r-x))}}{(r-x)^{1+\alpha}}\Bigg)f(r) \Bigg| \,dr \\ &\le 
	C(t)\int_{x}^{x+1} | f(r)| \,dr + 
	C(t) \int_{x+1}^\infty  \frac{| f(r)|}{(r-(x+ 1/2))^{2+\alpha}} \, dr.
\end{align*}
The absolutely convergence of both integrals follows from
Remarks \ref{positivity} and \ref{property1}. 
This allows us to pass the limit in $h$ inside the first integral. As for the second term, we observe 
$$\int_x^{x+h}\frac{e^{-t^2/(4(r-x))}}{(r-x)^{1+\alpha}}|f(r)|\,dr \le 
e^{-t^2/(8h)}\int_x^{x+h}\frac{e^{-t^2/(8(r-x))}}{(r-x)^{1+\alpha}}|f(r)|\,dr
\le C(t)e^{-t^2/(ch)}\int_x^{x+1}|f(r)|\,dr.$$
Aplying H\"older inequality and the local integrability of the weight (Remark \ref{positivity})  we get that this term multiplied by $1/h$ tends to $0$ as $h\to0$.
Pasting up  the last two thoughts we get 
\begin{align*}
	-D_{\rm right}ÊU(x,t) &= 
	 \frac{t^{2\s}}{4^\alpha\Gamma(\alpha)}
	\int_x^\infty \frac{d}{dx}\Bigg(\frac{e^{-t^2/(4(r-x))}}{(r-x)^{1+\alpha}}\Bigg)f(r)\,dr \\
	&= \frac{1}{4^\alpha\Gamma(\alpha)}
	\int_x^\infty\Bigg(\frac{d^2}{dt^2}+\frac{1-2\alpha}{t}\frac{d}{dt}\Bigg)
	\Bigg(\frac{t^{2\s}e^{-t^2/(4(r-x))}}{(r-x)^{1+\alpha}}\Bigg)f(r)\,dr \\
	&= U_{tt}(x,t)+\frac{1-2\alpha}{t}U_{t}(x,t).
\end{align*}
The last equality can again be justified by the absolutely convergence of the corresponding integrals.
Also, from \eqref{Poisson cambio} it is clear that $u(x,0)=f(x)$. To see \eqref{limit},
observe that
\begin{equation*}\label{zero int}
\int_0^\infty e^{-t^2/(4s)}\Bigg(2\a-\frac{t^2}{2s}\Bigg)\,\frac{ds}{s^{1+\a}}=0.
\end{equation*}
Given a smooth function $\varphi$, by Fubini's Theorem and  \eqref{gen Poisson} we get
\begin{align*}
 -c_\a\int_{\mathbb{R}}&t^{1-2\a}U_t(x,t)\varphi(x)\,dx  \\
 &= -\frac{4^{\a-1/2}\Gamma(\a)}{\Gamma(1-\a)}\frac{1}{4^\a\Gamma(\a)}
\int_{\mathbb{R}} \int_0^\infty e^{-t^2/(4s)}\Bigg(2\a-\frac{t^2}{4s}\Bigg)f(x+s)\,
\frac{ds}{s^{1+\a}}\,\varphi(x)\,dx  \\
&=  -\frac{1}{2\Gamma(1-\a)}
\int_{\mathbb{R}} \int_0^\infty e^{-t^2/(4s)}\Bigg(2\a-\frac{t^2}{4s}\Bigg)
\varphi(y-s)\,\frac{ds}{s^{1+\a}}\,f(y)\,dy \\
&= -\frac{1}{2\Gamma(1-\a)}
\int_{\mathbb{R}}\int_0^\infty e^{-t^2/(4s)}\Bigg(2\a-\frac{t^2}{4s}\Bigg)
\Big(\varphi(y-s)-\varphi(y)\Big) \,\frac{ds}{s^{1+\a}} \, f(y)\,dy.
\end{align*}
Therefore, 
$$-c_\a\lim_{t\to0^+}\int_{\mathbb{R}} t^{1-2\a}U_t(x,t)\varphi(x) \,dx
=\langle (D_{\rm left})^\a \varphi, \, f \rangle =\langle (D_{\rm right})^\a f,\varphi\rangle,$$
as we wanted to prove. 
To justify the interchange of the limit with the integral
we shall distinguish three cases. First, by using the Mean Value Theorem and the fact that $\varphi$
is in the Schwartz class, we can see that
$$\int_{s< |y|/2}  \Big|\varphi(y-s)-\varphi(y)\Big| \,\frac{ds}{s^{1+\a}}\le\frac{C_N}{(1+|y|)^N}.$$
for any $N$ large. On the other hand,
\begin{align*}
\int_{|y|/2 < s < 2|y|} \Big|\varphi(y-s)-\varphi(y)\Big| \,\frac{ds}{s^{1+\a}} 
& \le \int_{|y|/2 < s < 2|y|} |\varphi(y-s)|\,\frac{ds}{s^{1+\a}}
+ \int_{|y|/2 < s < 2|y|}|\varphi(y)| \,\frac{ds}{s^{1+\a}} \\
 &\le  \frac C{(1+|y|)^{1+\a}}  +C|\varphi(y)|.
 \end{align*}
Finally, 
$$\int_{s>2|y|} \Big|\varphi(y-s)-\varphi(y)\Big| \,\frac{ds}{s^{1+\a}} 
\le C_N \int_{s> |y|/2}\bigg|\frac1{1+|y-s|^N} -\frac1{1+|y|^N}\bigg| \,\frac{ds}{s^{1+\a}}
\le \frac{C_N}{1+|y|^N}.$$
An application of Remark \ref{rem:Silvestre} gives the conclusion.
\end{proof}

\section{Unified approach to several fractional operators defined by different authors}\label{unified}

Recall the definitions of the fractional integrals given in \eqref{kalpha}, \eqref{kalphaH} and \eqref{Halpha}.
Here is one of our unification results.

\begin{prop}\label{prop1}
Let $0< \alpha <1$. Then, for $f\in\mathcal{S}$, 
 \begin{itemize}
 \item[(i)] $\displaystyle k_\alpha^Hf = (-\Delta)^{-\alpha/2}f$;
 \item[(ii)] $\displaystyle k_\alpha f = \cot \tfrac{\pi \alpha}{2} (-\Delta)^{-\alpha/2}f$;
 \item[(iii)] $\displaystyle  H_\alpha f(x)=   \frac1{2\sin\frac{\pi \alpha}{2}} \Big( (D_{{\rm right}})^{-\alpha}f  -  (D_{\rm left})^{-\alpha}f \Big)$;
 \item[(iv)]   $H_\alpha f = (-\Delta)^{-\alpha/2} H f$,
 where $H$ denotes the classical Hilbert transform on $\R$.
\end{itemize}
\end{prop}

\begin{rem}\label{derivada7}
As we explained in the introduction we define $\displaystyle (-\Delta)^\beta =  \frac1{\Gamma(-\beta)}\int_0^\infty (e^{t\Delta}- Id) \frac{dt}{t^{1+\beta}} $ and $\displaystyle (-\Delta)^{-\beta} = \frac1{\Gamma(\beta)}\int_0^\infty e^{t\Delta} \frac{dt}{t^{1-\beta} }$, $0< \beta <1.$  
\end{rem}

\begin{proof}
By using (\ref{gamma1}) we have 
$$k_\alpha^H f= \frac{\Gamma(1-\alpha) \sin\frac12 \pi \alpha}{ \pi}
\frac{\pi^{1/2} 4^{\alpha/2} \Gamma(\alpha/2)}{\Gamma\Big(\frac{1-\alpha}{2}\Big)}
(-\Delta)^{-\alpha/2}f =  (-\Delta)^{-\alpha/2} f.$$
Also, 
\begin{align*}
 k_\alpha f &=  \frac1{2 \Gamma(\alpha) \sin \frac{\pi \alpha}{2} }
 \frac{\pi^{1/2} 4^{\alpha/2} \Gamma(\alpha/2)}{\Gamma\Big(\frac{1-\alpha}{2}\Big)}
 (-\Delta)^{-\alpha/2} f \\
&=  \frac {\pi }{ \sin\frac{\pi \alpha}{2} \Gamma( \frac{\alpha}{2} +\frac12)
\Gamma\Big(\frac{1-\alpha}{2}\Big)}  (-\Delta)^{-\alpha/2} f= \cot \tfrac{\pi \alpha}{2} (-\Delta)^{-\alpha/2}f .
\end{align*}
On the other hand,
\begin{align*}
  H_\alpha f(x) &=   \frac1{2\Gamma(\alpha) \sin \frac{\pi \alpha}{2} }
  \,  \bigg[\int_0^\infty \frac{f(x+u)}{u^{1-\alpha}}\, du -\int_0^\infty \frac{f(x-u)}{u^{1-\alpha}}\, du\bigg] \\
   &= \frac1{2\sin\frac{\pi \alpha}{2}} \Big( (D_{{\rm left}})^{-\alpha}f  -  (D_{\rm right})^{-\alpha}f \Big).
 \end{align*}
Finally, by using formulas (\ref{Fderivada}),
$$\widehat{H_\alpha f}(\xi) = \frac1{2\sin\frac{\pi \alpha}{2}}
\Big(  (-i\xi)^{-\alpha} - (i\xi)^{-\alpha} \Big) \widehat{f}(\xi)
=i\, {\rm sign}\,(\xi) |\xi|^{-\alpha} \widehat{f}(\xi).$$
\end{proof}

Next we will consider the fractional derivatives defined in \eqref{kmenosalpha}
and \eqref{Hmenosalpha}.
The following Theorem translates these operators to our language.

\begin{prop}\label{prop2}
Let $0<\alpha <1$. Then, for $u\in\mathcal{S}$,
\begin{itemize}
\item[(i)] $k_{-\alpha}u = (-\Delta)^{\alpha/2} u$;
\item[(ii)] $\displaystyle H_{-\alpha}u = -\frac1{2\sin\frac{\pi \alpha}{2}}
\big[(D_{\rm right})^\alpha u - (D_{\rm left})^\alpha u \big]$;
\item[(iii)] $H_{-\alpha}u =  H (-\Delta)^{\alpha/2} u,$ where $H$ is the Hilbert transform on $\R$.
\end{itemize}
\end{prop}

\begin{proof} By using (\ref{gamma1}) we have 
\begin{align*}
-\frac{\Gamma(1+\alpha) \sin\frac{\pi \a}{2}}{\pi} &=   -\frac{\Gamma(1/2+\alpha/2) \Gamma(1+\a/2)}
{2^{-\a} \sqrt{\pi}} \frac{\sin\frac{\pi \a}{2}}{\pi} \\
&=  \frac{\Gamma(1/2+\alpha/2) \Gamma(1+\a/2)}
{2^{-\a} \sqrt{\pi}} \frac{\sin\frac{\pi(- \a)}{2}}{\pi}  \\
&= \frac{\Gamma(1/2+\alpha/2) \Gamma(1+\a/2)}
{2^{-\a} \sqrt{\pi}\Gamma(-\a/2) \Gamma(1+\a/2)} 
= \frac{\Gamma(1/2+\a/2)}{2^{-\a} \sqrt{\pi} \Gamma(-\a/2)}.
\end{align*}
This gives (i). On the other hand, for (ii),
\begin{align*}
H_{-\a} u(x) &= \tfrac{1}{\pi}\Gamma(1+\a) \cos\tfrac{\pi \a}{2}
\bigg[\int_0^\infty \frac{u(x+t) -u(x)}{t^{1+\a}}\,dt
+\int_0^\infty \frac{u(x)-u(x-t)}{t^{1+\a}}\,dt\bigg] \\
&=\tfrac{1}{\pi}\Gamma(1+\a) \cos\tfrac{\pi \a}{2} \Gamma(-\a) 
\big[(D_{{\rm right}})^\alpha u(x)-  (D_{{\rm left}})^\a u(x)\big] \\
&=-\frac1{2\sin \frac{\pi \a}{2}}\big[(D_{{\rm right}})^\alpha u(x)-  (D_{{\rm left}})^\alpha u(x) \big].  
\end{align*}
Finally, the Fourier transform and \eqref{gamma1} produce
\begin{align*}
\widehat{H_{-\a} (u)}(\xi) &=
\tfrac{1}{\pi}\Gamma(1+\a) \cos\tfrac{\pi \a}{2} 
\Gamma(-\a)|\xi|^\a \bigg[\chi_{\xi>0} ((-i)^\a-(i)^\a)  
+\chi_{\xi<0} (i^\a-(-i)^\a)\bigg]\widehat{u}(\xi) \\ &=
-i\operatorname{sign}(\xi)\tfrac{1}{\pi}\Gamma(1+\a) \cos\tfrac{\pi \a}{2} 
\Gamma(-\a)|\xi|^\a 2\sin\tfrac{\pi \a}{2} \widehat{u}(\xi)\\
&= -i\operatorname{sign}(\xi)\tfrac{1}{\pi} \sin(\pi \a)  \Gamma(1+\a)\Gamma(-\a)  |\xi|^\a \widehat{u}(\xi)
 =  i\operatorname{sign}(\xi)|\xi|^\a\widehat{u}(\xi).
\end{align*}
\end{proof}

\section{Lateral Fractional Fundamental Theorem of Calculus}\label{Slateral}

We  begin this section by making some naive remarks about the composition of the operators
considered in Section \ref{powers}.
By combining Propositions \ref{prop1}  and \ref{prop2}, we see that 
the following compositions hold in $\mathcal{S}$:
\begin{itemize}
\item $k_{-\alpha}\circ k_\alpha^H f = (-\Delta)^{\alpha/2} \circ  (-\Delta)^{-\alpha/2} f = f$
\item $H_{\alpha} \circ H_{-\alpha} f = (-\Delta)^{-\alpha/2} H \circ H  (-\Delta)^{\alpha/2} f = -f$
\end{itemize}
On the other hand,  identities (\ref{derivada1}) and (\ref{derivada2}), together with their
Fourier transforms versions, imply the following Lemma.

\begin{lem}
For $f\in\mathcal{S}$, 
\begin{itemize}
\item[(i)] $ (D_{\rm right})^\alpha \circ (D_{\rm right})^{-\alpha} f
=  (D_{\rm right})^\alpha \circ W_\alpha f=  f$;
\item[(ii)] $ (D_{\rm left})^\alpha \circ (D_{\rm left})^{-\alpha} f
=  (D_{\rm left})^\alpha \circ R_\alpha f= f$.
\end{itemize}
\end{lem}

The definitions of $ (D_{\rm left})^\alpha$ and $ (D_{\rm right})^\alpha$
contain a singularity at the origin. Therefore the definition for general
functions has to contain a limit argument. In fact, we define 
$$(D_{\rm right} )_\varepsilon^\alpha u(x)
=\frac1{\Gamma(\alpha)} \int_\varepsilon^\infty \frac{u(x+t)-u(x)}{t^{\alpha+1}}\,dt
= \frac1{\Gamma(\alpha)} \int_{x+\varepsilon}^\infty \frac{u(t)-u(x)}{(t-x)^{\alpha+1}}\,dt.$$
Hence $(D_{\rm right} )^\alpha u(x) $ will be the limit
$\lim_{\varepsilon \rightarrow 0} (D_{\rm right} )_\varepsilon^\alpha u(x)$, whenever it exists. 

As we said in the Introduction,
a natural question is to know what is the best space for which the compositions above hold.
Heywood proved that the compositions are true for functions
$f\in L^p$,  $1<p< \frac1{\alpha}$.
In \cite{A} the author proved that the compositions 
are also satisfied for functions $f \in L^p(w^p)$,
$1<p < \frac1{\alpha}$, and $w$ in $A_{p,q}$, $0<\alpha<1$, and  $\frac{1}{q}=\frac{1}{p}-\alpha$.
We recall that a nonnegative weight function $w(x)$
defined on $\mathbb{R}$ is said to satisfy the $A_{p,q}$ condition 
if there exists a constant $C$ such that
\begin{equation}\label{claseApq}
\bigg( |I|^{-1} \int_I w^q\bigg)^{1/q}\bigg( |I|^{-1} \int_I w^{-p'}\bigg)^{1/{p'}}  \le C,
\end{equation}
for all intervals $I\subset \mathbb{R},$ where $|I|$ denotes the length of $I$.
As usual for $1<p<\infty$, $p'=p/(p-1)$. It is worth pointing out that if $1<p < \frac1{\alpha}$, $0<\alpha<1$, and  $\frac{1}{q}=\frac{1}{p}-\alpha$ then  the usual two-sided fractional integral operator $k_\a:L^p(w^p)\to L^q(w^q)$ is bounded if and only if $w(x)$
satisfies the $A_{p,q}$ condition.

As we observe in Remark \ref{rem:lateral}, the operators appearing
in the last lemma are lateral operators. Thus the classes of functions
for which the identities hold should contain an essential lateral argument.
We recall that a weight $w$ is in $A_{p,q}^+$, for $1<p,q<\infty$, if there exists a constant $C>0$ such that
\begin{equation}\label{lateral}
\bigg(\frac{1}{h} \int_{a-h}^a w^q \bigg)^{1/q}\bigg(\frac{1}{h} \int^{a+h}_a w^{-p'}\bigg)^{1/p'} \leq C,
\end{equation}
for all $a\in\R$ and all $h>0$,
see \cite{AS}, \cite{MPT} and \cite{MT}. Notice that $A_{p,q}\subset A_{p,q}^+$
and let us point out that the following characterization was obtained in \cite{AS}.

\begin{thm}[\cite{AS}, see also \cite{MPT} and \cite{MT}]
\label{acotfrac}
Let $0<\alpha<1$,  $1<p< \frac{1}{\alpha}$ and  $\frac{1}{q}=\frac{1}{p}-\alpha$. $W_\alpha:L^p(w^p)\to L^q(w^q)$ is bounded if and only if $w$ satisfies the one-sided
Muckenhoupt $A_{p,q}^+$ condition.
 The same characterization holds  for 
the one-sided fractional maximal operator $M_\alpha^+$ defined as 
$$M^+_\alpha f(x)=\sup_{h>0}\frac{1}{h^{1-\alpha}}\int_x^{x+h}|f(t)|\,dt.$$
(Notice that
$w$ is in $A_{p,q}^+$ if and only if $w^q$ satisfies the
$A_r^+$ condition, where $r=1+(q/p^\prime)$.)
\end{thm}

The lateral issue mentioned above is clarified in the following Theorems.
The first one is the precise version of Theorem \ref{thm:C}, while the second
one corresponds to Theorem \ref{thm:D}.

\begin{thm} \label{Inversionlateral2}
Let $0<\alpha<1$,  $1<p< \frac{1}{\alpha}$ and  $\frac{1}{q}=\frac{1}{p}-\alpha$.
If $w\in A_{p,q}^+$ then for all $f\in L^p(w^p)$,
$$\lim_{\varepsilon\to0^+}(D_{\rm right} )_\varepsilon^\alpha ((D_{\rm right} )^{-\alpha}f)(x)
=\lim_{\varepsilon\to0^+}(D_{\rm right} )_\varepsilon^\alpha (W_\alpha f)(x)= f(x),$$
in the almost everywhere sense and in the $L^p(w^p)$-norm.  
 \end{thm}

\begin{thm}\label{Inversionlateral3}
Let $0<\alpha<1$,  $1<p< \frac{1}{\alpha}$ and  $\frac{1}{q}=\frac{1}{p}-\alpha$.
Let  $w$ be a nonnegative measurable function on $\R$ such that $w^p$ and $w^{-p^\prime}$ are
locally integrable. Assume that $W_\alpha f(x) = (D_{\rm right} )^{-\alpha} f(x)<+\infty$ a.e. for all
 $f\in L^p(w^p)$.  Then
$$f(x)=\lim_{\varepsilon\to 0^+}(D_{\rm right} )_\varepsilon^\alpha (W_\alpha f)(x)\quad
\hbox{for a.e.}~x\in\R.$$
\end{thm}
 
We shall prove first Theorem \ref{Inversionlateral3}. Then Theorem \ref{Inversionlateral2}
follows easily.

Finally, as anticipated in the Introduction, we also characterize the range  of $W_\alpha$.

 \begin{thm} \label{Rango}
Let $0<\alpha<1$,  $1<p< \frac{1}{\alpha}$ and  $\frac{1}{q}=\frac{1}{p}-\alpha$.
If $w\in A_{p,q}^+$ and  $u$ is a measurable function on $\R$
then the following statements are equivalent.
\begin{itemize}
\item[(a)] There exists $f\in L^p(w^p)$ such that $u=W_\alpha f$.
\item[(b)] $u\in L^q(w^q)$ and there exists the limit
$\lim_{\varepsilon\to0^+}(D_{\rm right} )_\varepsilon^\alpha u(x)$ in the norm of $L^p(w^p)$.
\item[(c)] $u\in L^q(w^q)$ and
$\sup_{\varepsilon>0}||(D_{\rm right} )_\varepsilon^\alpha u||_{L^p(w^p)}<+\infty$.
\end{itemize}  
\end{thm}

\subsection{Some preliminaries}\label{prelims}
 
In order to prove the last statement we need a lemma (interesting on its own right)
which provides a convenient dense class in the one-sided weighted spaces
that is invariant for the fractional integrals
$I_\alpha$, $W_\alpha =(D_{\rm right} )^{-\alpha}$ and $R_\alpha = (D_{\rm left})^{-\alpha}$.
It is known that the Lizorkin class:
 $$\Phi:=\bigg\{\varphi\in \mathcal{S}: \frac{d^k\widehat\varphi}{dx^k}(0)=0, k\in\N\bigg\}=
 \bigg\{\varphi\in \mathcal{S}: \int{x^k}\varphi(x)\,dx=0, k\in\N\bigg\},$$
 is invariant for $I_\alpha$, $W_\alpha =(D_{\rm right} )^{-\alpha}$ and $R_\alpha = (D_{\rm left})^{-\alpha}$.
Furthermore, if $w$ is a weight in the Muckenhoupt $A_p$ class, $1<p<\infty$,
then $\Phi$ is dense in $L^p(w^p)$, see  \cite{NS}. That is also true for one-sided weighted spaces.

 \begin{lem} \label{Densidad}
 Let $0<\alpha<1$ and $1<s< \infty$. Assume that $w\in A_{s}^-$.
 Then $\Phi$ is dense in $L^s(w)$ and $(D_{\rm left})^{-\alpha}(\Phi)\subset \Phi$.
 \end{lem}
 
By reversing the orientation of the real line  we have the corresponding result for weights
in $A_s^+$ and $(D_{\rm right} )^{-\alpha}$.

\begin{proof}[Proof of Lemma \ref{Densidad}]
The proof follows the ideas in \cite{NS}.
It follows from \cite[Lemma 8.1,p. 148]{Sa} that $(D_{\rm left})^{-\alpha}(\Phi)\subset \Phi$ .
Let $C_c^\infty$ be the set  of smooth functions on $\R$ with compact support.
In order to prove the density of $\Phi$ it suffices to prove that $\Phi$ is dense in 
$C_c^\infty$ in the norm of $L^s(w)$.

Let $f \in C_c^\infty$. Rychkov \cite[\S4]{Ry} has proved that there exists a real function
$g\in \mathcal{S}$ supported in $[1,\infty)$ with the following properties:
$\int g=1$ and $\int x^kg(x)\,dx=0$, for all $k\in\N$,
or, equivalently,
$\widehat{g}(0)=1$ and $\frac{d^k\widehat{g}}{dx^k}(0)=0$, for all $k\in\N$.
For each $N\in\mathbb{N}$, let $g_N(x):=\frac1{N}g(\frac{x}{N})$ and 
$f_N:=f-g_N\ast f$. It is clear that $f_N\in \mathcal{S}$. Furthermore, 
$\widehat{f_N}(\xi)=(1-\widehat{g}(N\xi))\widehat{f}(\xi)$.
It follows that $f_N\in \Phi$. It is obvious that $\lim_{N\to\infty}g_N(x)=0$ for all $x$.
Since $g\in \mathcal{S}$ and is supported in $(0,\infty)$,
there exists an integrable function $F\ge 0$ with support in $(0,\infty)$ and decreasing in $(0,\infty)$
such that $g\leq F$. Then,  $|g_N\ast f|\leq CM^-f$, with $C=\int F$.
We know that $M^-f\in L^s(w)$ because $w\in A_s^-$. Therefore, by the dominated convergence theorem,
$\lim_{N\to\infty}\int|g_N\ast f|^sw=0.$
In other words, $\lim_{N\to\infty}f_N=f$ in the norm of $L^s(w)$, which proves the Lemma.
\end{proof}

The following lemmas will be crucial in the proofs. 
The first one follows from Proposition 2 and Lemma 6 in \cite[pp~953-- 954]{RS}.

\begin{lem}[\cite{RS}] \label{segovialiliana} 
Let $0<\alpha<1$,  $1<p< \frac{1}{\alpha}$ and  $p'=p/(p-1)$.
Let  $w$ be a nonnegative measurable function such that $w^{-p^\prime}$
is locally integrable.
The following statements are equivalent.
\begin{itemize}
\item[(a)] $(D_{\rm right})^{-\alpha}f(x)< +\infty$ a.e. for all nonnegative measurable functions $f\in L^p(w^p)$.
\item[(b)] There exists a pair of  real numbers $a<b$ such that
$$\int_b^\infty\frac{{w^{-p^\prime}}(y)}{(y-a)^{(1-\alpha)p^\prime}}\,dy<+\infty.$$
\item[(c)] For all pair of  real numbers $a<b$
$$\int_b^\infty\frac{{w^{-p^\prime}}(y)}{(y-a)^{(1-\alpha)p^\prime}}\,dy<+\infty.$$
\end{itemize}
\end{lem}

The following lemma follows from Theorem 3 in \cite{AS} by a very simple translation argument.

 \begin{lem}[\cite{AS}] \label{andersensawyer} Let   $1<p<+\infty$ and  $p'=p/(p-1)$.
 Let $a\in\R$ and  let  $v$ be a finite nonnegative measurable function on $(a,\infty)$.
 The following statements are equivalent.
\begin{itemize}
\item[(a)] There exist a positive measurable function $\omega$ on $(a,\infty)$ and a positive constant $C$  such that 
$$\int_a^\infty |M^+f|^p\omega^p\leq C\int_a^\infty |f|^pv^p,$$
for all measurable functions $f$.
\item[(b)]For all  $b>a$
$$\sup_{S>b-a}\frac1{S^{p^\prime}}\int_{b}^{a+S}{v^{-p^\prime}(y)}\,dy<+\infty.$$
\end{itemize}
 \end{lem}
 
 \begin{rem}\label{remark}
 {\it Note that (c) in Lemma \ref{segovialiliana} implies (b) in Lemma \ref{andersensawyer}. In fact,  
 if $S>b-a$ then
$$
 \aligned
 \frac1{S^{p^\prime}}\int_{b}^{a+S}{w^{-p^\prime}(y)}\,dy
 &\leq\frac{1}{(b-a)^{\alpha p^\prime}}
 \frac{1}{S^{(1-\alpha)p^\prime}} \int_{b}^{a+S}{w^{-p^\prime}(y)}\,dy\\
&\leq
\frac{1}{(b-a)^{\alpha p^\prime}}
 \int_{b}^{a+S}\frac{w^{-p^\prime}(y)}{(y-a)^{(1-\alpha)p^\prime}}\,dy.
 \endaligned$$}
  \end{rem}
  
\subsection{Proofs of Theorems \ref{Inversionlateral2} and \ref{Inversionlateral3}}

\begin{proof}[Proof of Theorem  \ref{Inversionlateral3}] 
We follow the ideas   in \cite[\S6]{Sa}.

It suffices to prove that for fixed $a\in\R$, $f\in L^p(w^p)$, $f\ge 0$, we have that  
\begin{equation}
\label{formula}
\lim_{\varepsilon\to0^+}(D_{\rm right} )_\varepsilon^\alpha ((D_{\rm right} )^{-\alpha} f)(x)=f(x)
\end{equation}
for a.e. $x>a$.
Let $A_f=\{x>a:W_\alpha f(x)<+\infty\}$ and $B_f=\{x>a: M^+f(x)<+\infty\}$.
By the assumption  and by Remark \ref{remark} we have
that statement (a) in Lemma  \ref{andersensawyer} holds and, consequently,
$|(a,\infty)\setminus A_f\cap B_f|=0$.

We are going to prove that \eqref{formula} holds for $x\in A_f\cap B_f$
(consequently, for almost every  $x>a$).

\smallskip

\noindent
\textbf{Claim.} For all $x\in A_f\cap B_f$,
 $$(D_{\rm right} )_\varepsilon^\alpha (W_\alpha f)(x)=\frac{1}{\Gamma(-\alpha)}
\int_{-\infty}^0 f(x-y) \tfrac{1}{\varepsilon} \tilde{k}(\tfrac{y}{\varepsilon})\, dy =
 f\ast\tilde{k}_\varepsilon(x),$$
 where, for $x<0$,  $\displaystyle\tilde{k}(x)= \frac{1}{\Gamma(-\alpha)|x|} \int_{x}^{0} k(r)\, dr$,
 with
 \begin{equation}\label{eq:k}
 k(r)= \frac{1}{\Gamma(\alpha)} \left[|r+1|^{\alpha-1}\chi_{(-\infty,-1)}(r)-|r|^{\alpha-1}
\chi_{(-\infty,0)}(r)\right],
\end{equation}
 and $\tilde{k}(x)=0$ for $x>0$.

\smallskip

Using this claim, the proof of \eqref{formula} is straightforward. We observe that
\begin{equation}\label{eq:k tilde}
\tilde{k}(x)= \frac{-1}{\alpha\Gamma(-\alpha) \Gamma(\alpha)}
\left[ \frac{|x+1|^{\alpha}-|x|^{\alpha}}{|x|} \chi_{(-\infty,-1)}(x) -|x|^{\alpha-1}  \chi_{(-1,0)}(x)\right].
\end{equation}
Since $\tilde{k}$ is integrable with $\int \tilde{k}=1$ (see Lemma \ref{11}), has support in
$(-\infty,0)$,  is increasing and nonnegative in that interval (see \cite{Sa}), we obtain
by the results regarding approximations of the identity
(Subsection \ref{approx})
that for a.e.  $x\in A_f\cap B_f$
$$(D_{\rm right} )_\varepsilon^\alpha (W_\alpha f)(x)=f\ast\tilde{k}_\varepsilon(x)\leq M^+f(x),$$
and
$$\lim_{\varepsilon\to 0^+}
(D_{\rm right} )_\varepsilon^\alpha (W_\alpha f)(x)=\lim_{\varepsilon\to 0^+}
 f\ast\tilde{k}_\varepsilon(x)=f(x),$$
almost everywhere because statement (a) in  Lemma \ref{andersensawyer} holds (see Remark \ref{remark}).
\end{proof}

 \begin{proof}[Proof of the Claim]
We follow the proof in \cite[Chapter~2,~Lemma~6.1,~p.~184]{Sa}.
The difficulties and differences appear because we are working in weighted spaces
and the computations must be properly justified.

Let us fix $x\in A_f\cap B_f$.  The function
 $h_x(t):= W_\alpha f(x+t)- W_\alpha f(x)$
 is defined for a.e. $t$. We have
\begin{equation*}\label{cuenta1}
\begin{aligned}
W_\alpha f(x+t)- W_\alpha f(x)
&= \frac{1}{\Gamma(\alpha)} \int_0^\infty \frac{f(x+t+s)}{s^{1-\alpha}}\, ds - \frac{1}{\Gamma(\alpha)} \int_0^\infty \frac{f(x+s)}{s^{1-\alpha}}\, ds\\
&= \frac{t^\a}{\Gamma(\alpha)}\int_{-\infty}^{-1} \frac{f(x-tr)}{|r+1|^{1-\alpha}}\, dr
- \frac{t^\a}{\Gamma(\alpha)} \int_{-\infty}^0 \frac{f(x-tr)}{|r|^{1-\alpha}}\, dr\\
&= t^{\alpha}\int_{-\infty}^0 f(x-tr) k(r)\, dr,
\end{aligned}
\end{equation*}
where $k(r)$ is as in \eqref{eq:k}. Let us see that for all $\varepsilon>0$, 
\begin{equation*}\label{Fubini}
(D_{\rm right} )_\varepsilon^\alpha (W_\alpha f)(x)=
\frac{1}{\Gamma(-\alpha)}\int_{\varepsilon}^\infty
\frac1t\int_{-\infty}^0f(x-tr)k(r)\,dr\,dt
\end{equation*} 
is well defined. In other words, we are going to show that the following quantity
\begin{equation*}\label{tonelli}
\begin{aligned}
 \int_{\varepsilon}^\infty
\frac1t\int_{-\infty}^0f(x-tr)|k(r)|\,dr\,dt
&= \int_{\varepsilon}^\infty
\frac1{t^2}\int_{-\infty}^0f(x-y)\left|k\left(\frac{y}{t}\right)\right|\,dy\,dt \\
&=\int_{-\infty}^0
f(x-y)\int_{\varepsilon}^\infty
\frac1{t^2}\left|k\left(\frac{y}{t}\right)\right|\,dt\,dy \\
&= \int_{-\infty}^0
f(x-y)\frac{1}{|y|}\int_{y/\varepsilon}^0
|k(r)|\,dr\,dy,
\end{aligned}
\end{equation*}
 is finite. Let us  define $\displaystyle\overline k(x)=\frac{1}{|x|}\int_x^0|k(r)|\,dr$, for $x<0$, and
 $\overline k(x)=0$, for $x>0$. Then we have to show that
$$I:=\int_{-\infty}^0f(x-y)\tfrac{1}{\varepsilon}\overline k(\tfrac{y}{\varepsilon})\,dy<+\infty.$$
Note that for $x<0$,
$$\aligned
\bar{k}(x)&= C(\alpha) \frac{2+|x+1|^{\alpha}-|x|^{\alpha}}{|x|}\chi_{(-\infty,-1)}(x)+C(\alpha) |x|^{\alpha-1} \chi_{(-1,0)}(x)\\
&= \bar{k}_1(x)+\bar{k}_2(x).
\endaligned$$
Then $I=I_1+I_2$, where $I_i$ is the integral against the kernel $\bar{k}_i$, $i=1,2$.
Using that $\bar{k}_1 (x) \leq\frac{c}{|x|}$, for $x<-1$, we have that
$$\aligned
I_1 &\leq C\int_{-\infty}^{-\epsilon} |f(x-y)| \frac{1}{|y|}\, dy
\leq C ||f||_{L^p(w^p)} \left(\int_{-\infty}^{-\epsilon} \frac{w^{-p'}(x-y)}{|y|^{p'}}\,dy\right)^{1/p'}\\
&=||f||_{L^p(w^p)} \left(\int_{x+\epsilon}^{\infty} \frac{w^{-p'}(z)}{(z-x)^{p'}}\,dz\right)^{1/p'}.
\endaligned
$$
Now note the las integral is dominated by 
$$
\left(\frac{1}{\varepsilon^{p^\prime}} \left(\int_{x+\epsilon}^{x+\epsilon+1} w^{-p'}(z)\,dz\right)^{1/p'}
+
\left(\int_{x+\epsilon+1}^{\infty} \frac{w^{-p'}(z)}{(z-x)^{(1-\alpha)p'}}\,dz\right)^{1/p'}
\right)
<\infty.
$$
In the last inequality we use Lemma \ref{segovialiliana} 
and the local integrability of $w^{-p^\prime}$.
To estimate $I_2$ we 
apply the results on approximation of the identities to the kernel $\bar{k}_2$.
Therefore 
 $$I_2\leq\left(\int\bar{k}_2\right) M^+f(x)<+\infty.$$ 
 Once we have that $(D_{\rm right} )_\varepsilon^\alpha (W_\alpha f)(x)$ is well defined for $x\in A_f\cap B_f$, we apply Fubini's theorem in \eqref{Fubini} and we get the claim.
\end{proof}

\begin{proof}[Proof of Theorem  \ref{Inversionlateral2}]
By Remark \ref{positivity}, we 
know that  $w^{-p^\prime}$ is locally integrable.
By Theorem \ref{acotfrac}
we have that
$(D_{\rm right})^{-\alpha}:L^p(w^p)\to L^q(w^q)$ is bounded. Therefore, 
$(D_{\rm right})^{-\alpha}f(x)<\infty$ a.e. for all $f\in L^p(w^p)$.
As before, we may assume that $f\ge 0$. By what we have already seen in the proof of 
Theorem  \ref{Inversionlateral3}, for all $f\in L^p(w^p)$, $f\ge 0$,
$$\lim_{\varepsilon\to 0^+}
(D_{\rm right} )_\varepsilon^\alpha (W_\alpha f)(x)=
 f\ast\tilde{k}_\varepsilon(x)=f(x), $$
in the a.e. sense. Furthermore,
$f\ast\tilde{k}_\varepsilon(x)\leq M^+f(x).$
It is easily seen that $w\in A_{p,q}^+$ implies $w^p\in A_p^+$.
Therefore, $M^+:L^p(w^p)\to L^p(w^p)$
is bounded and, consequently,  $(D_{\rm right} )_\varepsilon^\alpha (W_\alpha f)(x)$
converges to $f(x)$ in the norm of $L^p(w^p)$.
\end{proof}

\begin{lem}\label{11}
Let $\tilde{k}(x)$ be as in \eqref{eq:k tilde}.
Then $\displaystyle \int_{\mathbb{R}} \tilde{k}(x) dx =1$.
\end{lem}
\begin{proof}
Observe that
\begin{align*}
\int_1^R&\frac{(r-1)^\alpha-r^\alpha}{r}\,dr - \int_0^1 \frac{r^\alpha}{r}\, dr = 
  \int_1^R\frac{(r-1)^\alpha}{r}\,dr -\int_0^R \frac{r^\alpha}{r}\,dr  \\
&=  \int_1^R\frac{(r-1)^\alpha}{r}\,dr -\int_1^{R}
 \frac{(r-1)^\alpha}{(r-1)}\,dr-\frac{R^\a-(R-1)^\a}{\a}  \\
 &=   -\int_1^R \frac{(r-1)^{\alpha-1} }{r} \, dr -\frac{R^\a-(R-1)^\a}{\a} 
= -\int_0^{R-1}\frac{z^{\alpha-1}}{z+1}\, dz-\frac{R^\a-(R-1)^\a}{\a}.
\end{align*}
The Lemma follows by noticing that
$\displaystyle-\lim_{R\rightarrow \infty} \int_0^{R-1}\frac{z^{\alpha-1}}{z+1}\,dz= 
\mathrm{B(\a,1-\a)}=\Gamma(\a)\Gamma(1-\a)$, where $\mathrm{B}(x,y)$
denotes the Beta function, see \cite{Lebedev}.
\end{proof}
\subsection{Proof of Theorem \ref{Rango}}

(a)$\Rightarrow$(b) and (a)$\Rightarrow$(c) are  consequences of Theorem \ref{Inversionlateral3}, its proof and the characterization of the boundedness of the one-sided fractional integral \cite{AS}. 

(b)$\Rightarrow$(a) Let $f=\lim_{\varepsilon\to0^+}(D_{\rm right} )_\varepsilon^\alpha u$
 in the norm of $L^p(w^p)$. On one hand,
 the operator $(D_{\rm right})^{-\alpha}:L^p(w^p)\to L^q(w^q)$ is bounded because $w$
 is in $A_{p,q}^+$. Then
 $(D_{\rm right})^{-\alpha}f=\lim_{\varepsilon\to0^+}(D_{\rm right})^{-\alpha}((D_{\rm right} )_\varepsilon^\alpha u)$
 in the norm of $L^q(w^q)$.  On the other hand,   
 $w\in A_{p,q}^+\Rightarrow w^q\in A_q^+\Leftrightarrow w^{-q^\prime}\in A_{q^\prime}^-$
 and, consequently, by Lemma \ref{Rango}, $\Phi$ is dense in $L^{q^\prime}(w^{-q^\prime})$.
 Next let us fix  $\varphi\in \Phi$. As usual $\langle f,g\rangle$ denotes the integral $\int fg$. Then 
 $$ \langle(D_{\rm right})^{-\alpha}f,\varphi\rangle=\lim_{\varepsilon\to0^+}
 \langle(D_{\rm right})^{-\alpha}((D_{\rm right} )_\varepsilon^\alpha u),\varphi\rangle
 =\lim_{\varepsilon\to0^+}
 \langle(D_{\rm right} )_\varepsilon^\alpha u,(D_{\rm left})^{-\alpha}\varphi\rangle.$$
 Let $h=(D_{\rm left})^{-\alpha}\varphi$. Note that $h\in \Phi$ since
 $\varphi\in \Phi$. In particular $h\in L^{q^\prime}(w^{-q^\prime})$. Therefore, $uh$ is integrable.
 By Fubini's Theorem,
 \begin{equation}\label{eq:Fubini}
 \langle(D_{\rm right} )_\varepsilon^\alpha u,h\rangle=\langle u,(D_{\rm left})_\varepsilon^{\alpha}h\rangle.
 \end{equation}
 We shall justify the application of Fubini's Theorem at the end of the proof of the implication.
 Therefore,
$$ \langle(D_{\rm right})^{-\alpha}f,\varphi\rangle=\lim_{\varepsilon\to0^+}
 \langle u,(D_{\rm left})_\varepsilon^{\alpha}h\rangle=\lim_{\varepsilon\to0^+}
 \langle u,(D_{\rm left})_\varepsilon^{\alpha}((D_{\rm left})^{-\alpha}\varphi)\rangle. $$
 By the analogue of Theorem \ref{Inversionlateral3} for left one-sided weights,
 $\lim_{\varepsilon\to0^+}
 (D_{\rm left})_\varepsilon^{\alpha}((D_{\rm left})^{-\alpha}\varphi)=\varphi$
  in the norm of $L^{q^\prime}(w^{-q^\prime})$. Since $u\in L^q(w^q)$, we finally get
$ \langle(D_{\rm right})^{-\alpha}f,\varphi\rangle=\langle u,\varphi\rangle$,
 for all $\varphi\in \Phi$. Since $\Phi$ is dense in $L^{q^\prime}(w^{-q^\prime})$,
 $(D_{\rm right})^{-\alpha}f,u\in L^q(w^q)$
 and $L^q(w^q)$ is the dual of  $L^{q^\prime}(w^{-q^\prime})$
 we obtain that $(D_{\rm right})^{-\alpha}f=u$.

Let us justify the application of Fubini's Theorem in \eqref{eq:Fubini}.
As we said, $uh$ is integrable. Therefore we only have to show that
 $$A=\int_\R\left(\int_{x+\varepsilon}^\infty\frac{|u(t)|}{(t-x)^{1+\alpha}}\,dt \right)|h(x)|\,dx<\infty.$$
 Observe that the kernel $k(x):=|x|^{-(1+\alpha)}\chi_{(-\infty,-\varepsilon)}(x)$ is integrable,  supported and increasing in $(-\infty,0)$. Then, by  \eqref{identity2},
 $$ A\leq \frac1{\alpha\varepsilon^\alpha}\int_\R |h(x)|M^+u(x)dx. $$
 We know that $h\in L^{q^\prime}(w^{-q^\prime})$. Furthermore, $M^+u\in L^q(w^q)$ since 
 $u\in L^q(w^q)$ and  $w^q\in A_q^+$ (recall that $M^+$ is bounded in
 $L^q(w^q)$). It follows that the last integral is finite.

 (c)$\Rightarrow$(a). The assumption implies the existence of a sequence $\varepsilon_k\to 0^+$ such that
 there exists $\lim_{k\to\infty}(D_{\rm right})^\a_{\varepsilon_k}u$ in the norm of $L^p(w^p)$.
 Then the proof follows as before but working only with that sequence.

\medskip

\noindent\textbf{Acknowledgments.} We are very grateful to the referee for valuable comments
that helped us to improve the presentation of the paper.



\end{document}